\documentclass[reqno,a4paper,12pt]{amsart}
\usepackage{amsmath,amsthm,amsfonts,amssymb}
\usepackage{verbatim, graphicx, ifthen, enumitem}
\usepackage[T1]{fontenc}

\allowdisplaybreaks

\let\oldmarginpar\marginpar
\renewcommand\marginpar[1]{\-\oldmarginpar[\raggedleft\footnotesize #1]%
{\raggedright\footnotesize #1}}
\newtheorem{theorem}{Theorem}
\newtheorem{lemma}{Lemma}

\theoremstyle{definition}
\newtheorem{definition}{Definition}

\theoremstyle{remark}
\newtheorem{remark}{Remark}

\newcommand{\Z}{\mathbb{Z}}

\newcommand{\abs}[1]{|#1|}
\newcommand{\Abs}[1]{\left|#1\right|}

\newcommand{\N}{\mathbb{N}}
\newcommand{\R}{\mathbb{R}}

\newcommand{\C}{\mathbb{C}}

\def\l{\lambda}

\def\N{\mathbb{N}}
\def\Z{\mathbb{Z}}

\def\R{\mathbb{R}}
\def\C{\mathbb{C}}

\def\1{\mathbf{1}}
\def\eps{\varepsilon}
\def\a{\alpha}
\def\b{\beta}

\newcommand{\dif}{\mathrm{d}}
\newcommand{\e}{\mathrm{e}}
\newcommand{\im}{\mathrm{i}}
\newcommand{\norm}[1]{\|#1\|}

\begin{document}

 \title[From exact systems to Riesz bases in the Balian--Low theorem]{From exact systems to Riesz bases \\ in the Balian--Low theorem}

\author{Shahaf Nitzan}
\address{Department of Mathematical Sciences, Norwegian University of
Science and Technology (NTNU), NO-7491 Trondheim, Norway}
\email{sigal.nitzan@math.ntnu.no}

\author{Jan-Fredrik Olsen}
\address{Centre for Mathematical Sciences, Lund University, P.O. Box 118, SE-221 00 Lund, Sweden}
\email{janfreol@maths.lth.se}

\thanks{The first author is supported by the ERCIM ``Alain Bensoussan'' fellowship nr. 2009-05.}
\thanks{For a part of this work, the second author is
supported by the Research Council of Norway grant 160192/V30.}


\begin{abstract}
We look at the time--frequency localisation of generators of lattice Gabor systems.
For a generator of a Riesz basis, this localisation is described by the classical Balian--Low theorem.
We establish Balian--Low type theorems for complete and minimal Gabor systems with a frame-type approximation property. 
These results describe how the best possible localisation of a generator is limited by the degree of control over the coefficients  in approximations given by the system, and provide a continuous transition between the classical Balian--Low
conditions and the corresponding
conditions for  generators of complete and minimal systems. 
Moreover, this holds for the non-symmetric
generalisations of these theorems as well.
\end{abstract}

\maketitle

\section{introduction}

For $g\in L^2(\R)$ and $a,b > 0$, the Gabor system generated by $g$ on the lattice $a\Z \times b\Z$ is denoted by
\begin{equation*}
    G(g,a,b):=\{\e^{2\pi \im b m t}g(t-an)\}_{(m,n) \in \Z^2}.
\end{equation*}
Such systems were considered by Gabor \cite{gabor1946theory} and
today they play a prominent role in time--frequency analysis and
its applications \cite{dorfler2001, grochenig1999, kozek1998, strohmer2006}.
An interesting general problem in Gabor analysis is to find ``optimal''  bounds for the time--frequency localisation of the window function $g$, given appropriate constraints on the desired Gabor system. The Balian--Low theorem  gives a precise solution to a version of this problem for Riesz bases.
%

In the context of this paper, a system $G(g,a,b)$ is considered ``good'' if it is at least exact  (i.e., complete and minimal).
It is ``better'' if, in addition, 
the coefficients in the approximations it provides can be, in some sense, controlled (e.g., it is a Riesz basis).
Our main objective is to study 
the time--frequency
localisation of $g$ for
a scale of systems that
lie between Riesz
bases  and exact systems\footnote{In some sense, the Gaussian $g = \exp(-x^2/2)$  has the best possible time--frequency localisation. In this case, the system $G(g,a,b)$ is a frame in $L^2(\R)$ if and only if $ab < 1$ \cite{seip92}. However, it is always over-complete, i.e., it is never exact. The construction of Gabor orthonormal bases or exact systems is more delicate.}.
We  extend an uncertainty principle, known as the Balian--Low theorem, to these systems.

Since we are interested in exact systems, we consider systems $G(g,a,b)$ with $a=b=1$.
This is due to the known fact that if $ab>1$, then
$G(g,a,b)$ is not complete in $L^2(\R)$, while if $ab<1$, it is not
a minimal system there. (See \cite{ramanathan_steger95} for the
first claim. A modification of the same argument gives the second
claim.). Our results can be extended to any lattice $a\Z \times
b\Z$ with $ab=1$ by an appropriate dilation of the generating
function $g$.

\subsection{Balian--Low type theorems for Riesz bases and exact systems} \label{balian-low type section}
The Balian--Low theorem
\cite{balian1981,daubechies1990, low1985} is a manifestation of the uncertainty
principle in the context of Gabor analysis. It states that if the
system $G(g,1,1)$ is a Riesz basis in $L^2(\R)$, then the generator
$g$ must have much worse time--frequency localisation than allowed
by the uncertainty principle. More precisely, if $r\geq 2$, then at
least one of the integrals
\begin{equation}
    \label{e:bl cond for rb, r-r}
     \int_\R |\xi|^{r}|\hat{g}(\xi)|^2\dif \xi, \qquad \qquad \qquad  \int_\R |t|^{r}|g(t)|^2 \dif t,
\end{equation}
must diverge, where $\hat{g}$ denotes the Fourier transform of
$g$. This result is sharp. That is, for any $r<2$ there exists a
function $g\in L^2(\R)$ such that $G(g,1,1)$ is a Riesz basis and
both of the integrals in \eqref{e:bl cond for rb, r-r} converge
\cite{benedetto_czaja_gadzinski_powell2003}.

Among the systems that we consider, Riesz bases have the best
properties while exact systems have the weakest. For
the latter, a Balian--Low type theorem was established
in \cite{daubechies_janssen93}. It states that if the system
$G(g,1,1)$ is exact and $r\geq 4$, then at least one of the
integrals in \eqref{e:bl cond for rb, r-r} must diverge. It follows from our Theorem \ref{symmetric theorem}
below that this result is sharp.

More generally, non-symmetric time--frequency conditions for
generators of such systems have been considered. Namely, for which $r$
and $s$ can both of the integrals
\begin{equation}
    \label{e:bl cond r-s}
     \int_\R |\xi|^{r}|\hat{g}(\xi)|^2\dif \xi, \qquad \qquad \qquad  \int_\R |t|^{s}|g(t)|^2 \dif t,
\end{equation}
converge?  For simplicity of formulations, and without loss of generality, we assume that $r \leq s$ when discussing this
case.
For generators of Riesz bases, non-symmetric conditions were found in
\cite{benedetto_czaja2006, gautam08, grochenig96}: If $G(g,1,1)$ is a
Riesz basis and
\begin{equation}
\label{e:bl cond for rb, c-d}
 \frac{1}{r}+\frac{1}{s}\leq 1,
\end{equation}
 then at least one of the integrals in \eqref{e:bl cond r-s} must diverge.
As above, this result is sharp
\cite{benedetto_czaja_gadzinski_powell2003} (see also
\cite{benedetto_czaja_powell2006}).

The non-symmetric conditions for generators of exact systems were
studied in \cite{heil_powell09}, where it is shown that if the
system $G(g,1,1)$ is exact, and $ r \leq s$ satisfy
\begin{equation}
    \label{e:bl cond for exact, c-d}
    \frac{3}{r}+\frac{1}{s}\leq 1,
\end{equation}
 then at least one of the integrals in \eqref{e:bl cond r-s} must
 diverge. Again, the fact that this result is sharp follows from our
 Theorem \ref{main theorem} below.

The results presented in this paper provide a continuous 
interpolation between the condition $r\geq 2$ for generators of Riesz bases,
and the condition $r\geq4$ for generators of exact systems
mentioned above. Moreover, these results are extended to
the general non-symmetric case, where a similar interpolation is
given between the conditions in \eqref{e:bl cond for rb,
 c-d} and in \eqref{e:bl cond for exact, c-d}. In all cases, the results 
 are sharp.

To obtain these results, we develop some
 new insights into the connection between the time--frequency
 localisation of a function and the smoothness of its
 Zak transform (see Section \ref{smoothness section}).

\subsection{Between Riesz bases and exact systems}
We now describe the family of systems that we consider in this
work.

 Let $H$
be a separable Hilbert space. A system $\{f_n\}$ is a
Riesz basis in $H$ if it is an exact frame, i.e, if it is exact and the
following inequality holds for every $f\in H$:
\begin{equation}
    \label{e:frame cond}
    A\|f\|^2\leq \sum|\langle f, f_n\rangle|^2 \leq B\|f\|^2,
\end{equation}
where $A$ and $B$ are positive constants not depending on $f$.

In most cases, the right-hand side inequality in \eqref{e:frame
cond}, the \textit{Bessel} property, holds automatically.
Therefore, if one wants to relax the frame condition, there is
usually no advantage in changing it. The left-hand side inequality
in \eqref{e:frame cond} is equivalent to \emph{completeness with
$\ell^2$ control over the coefficients}: Every $f\in H$ can be
approximated, with arbitrary small error, by a finite linear
combination $\sum a_nf_n$ with
$\sum |a_n|^2\leq C\|f\|^2$, for some positive constant $C$ not depending  on $f$. 
We are interested in exact systems with a relaxed version of this property. We use the following definition introduced in \cite{nitzan_olevskii07}.

\begin{definition} Given $q\geq 2$, we say that a system
$\{f_n\}$ is a $(C_q)$-system in $H$ (\emph{complete with $\ell^q$
control over the coefficients}) if every $f\in H$ can be
approximated, with an arbitrary small error, by a finite linear
combination $\sum a_nf_n$ with
\begin{equation*}
     \Big(\sum|a_n|^q \Big)^{\frac{1}{q}} \leq C\|f\|,
\end{equation*}
where $C = C(q)$ is a positive constant not depending on $f$.
\end{definition}

Note that all $(C_q)$-systems are complete. In addition, if
$q_1\leq q_2$, then a $(C_{q_1})$-system is also a
$(C_{q_2})$-system. Thus, we obtain a range of systems which
become ``better'' the closer $q$ is to $2$. In this extreme
case, a system is a Bessel $(C_2)$-system if and only if it is a
frame. The following dual formulation \cite{nitzan_olevskii07}
enhances the analogy between $(C_q)$-systems and frames: A system
is a $(C_q)$-system if and only if
\[
c\|f\| \leq \Big(\sum |\langle f,f_n\rangle |^p \Big)^{\frac{1}{p}}, \qquad
\forall f\in H,
\]
 where $1/p+1/q=1$ and $c=c(p)$ is a positive constant not depending on $f$.
 This condition should be compared with the left
inequality in \eqref{e:frame cond}.

 In general, frames and $(C_q)$-systems are not exact. As a system is a Riesz basis if and only if
 it is an exact frame, exact $(C_q)$-systems (which are also Bessel systems) can be considered relaxed forms of Riesz
bases. See Theorem 3 in this
context. In particular, it follows from this theorem that if a
Gabor system $G(g,1,1)$ is exact, then it is also a $(C_{\infty})$-system. 
Therefore, in this case, exact $(C_q)$-systems provide a
continuous scale of systems, ranging from Riesz bases ($q=2$) to
exact systems ($q=\infty$).

\subsection{The main result}
Our main result is Theorem \ref{main theorem}.
We first discuss a simplified version of it; the symmetric case,
where the localisation conditions for the generator $g$ are the
same in time and in frequency.

\begin{theorem}[The symmetric case] \label{symmetric theorem}
Fix $q>2$.
\begin{enumerate}[label=$(\alph*)$]
\item Let $g \in L^2(\R)$ and $r > 4(q-1)/q$. If $G(g,1,1)$ is an
exact $(C_q)$-system in
 $L^2(\R)$, then at least one of the integrals in \eqref{e:bl cond for rb,
 r-r}
   must diverge.
\item Let $r < 4(q-1)/q$. There exists a function $g \in L^2(\R)$
for which $G(g,1,1)$ is an exact $($Bessel$)$ $(C_q)$-system in
$L^2(\R)$ while both integrals in \eqref{e:bl cond for rb,
 r-r} converge.
\end{enumerate}
\end{theorem}
Note that when $q=2$, we have $4(q-1)/q=2$. Recall that an exact
(Bessel) $(C_2)$-system is a Riesz basis, and so, in this case, Theorem
\ref{symmetric theorem} should be compared with the classical
Balian--Low theorem (which studies the condition $r\geq 2$). On the
other hand, when $q$ tends to infinity, then $4(q-1)/q$ tends to
$4$, which is the best localisation possible for generators of
exact systems according to the corresponding Balian--Low type theorem  (which studies the condition $r\geq 4$). In
particular, part $(b)$ of Theorem \ref{symmetric theorem} implies
that this result   is sharp.

To state our main result in full generality, we introduce some
notation. For $2\leq q\leq \infty$, denote by $\Gamma_q$
the restriction to the area $0 \leq v \leq u $ of the curve
determined by the equations
\begin{equation}
	\begin{split} \label{gamma q}
    \frac{3q-2}{q+2}\cdot u+v =1, \qquad &u+3v\leq 1, \\
         u + v= \frac{q}{2(q-1)}, \qquad &u+3v> 1.
         \end{split}
\end{equation}
This curve corresponds to EFG in Figure \ref{figure}.
As can be seen in the figure,  the sector $0\leq v\leq u$ can be written as a partition
\begin{equation*}
    S_q \cup \Gamma_q \cup W_q; \qquad (0,0)\in S_q, (1,1)\in W_q,
\end{equation*}
where $S_q$ is represented by the shaded area below the curve
$\Gamma_q$ in the Figure, and $W_q$ by the non-shaded area above
the curve.
Accordingly, we say that a point $(u,v)$ in the sector $0 \leq v
\leq u$ lies either \textit{above}, \textit{on},  or
\textit{below} $\Gamma_q$.

We note that in the area
discussed, the curve $\Gamma_2$ is given by $u+v=1$ and the curve
$\Gamma_{\infty}$ by $3u+v=1$, (the segments BC and AD in Figure
\ref{figure}, respectively). In Theorem \ref{main theorem} the areas
below these curves represent the conditions for Riesz bases and
exact systems given in \eqref{e:bl cond for rb,
 c-d} and \eqref{e:bl cond for exact, c-d}.

We are now ready to state and discuss our main result. 
\begin{theorem} \label{main theorem} \label{main result}
   Fix $q>2$ and let $\Gamma_q$ be as above.
\begin{enumerate}[label=$(\alph*)$]
\item Let $g \in  L^2(\R)$ and $r \leq s$ be such that the point
   $(\frac{1}{r},\frac{1}{s})$ is below the curve $\Gamma_q$. If $G(g,1,1)$ is an exact $C_q$-system in $L^2(\R)$,
    then at least one of the integrals in $\eqref{e:bl cond r-s}$
    must diverge.
\item Let $r \leq s$ be such that the
point $(\frac{1}{r},\frac{1}{s})$ is above the curve $\Gamma_q$.
    Then there exists a function $g \in L^2(\R)$ for which $G(g,1,1)$ is an exact $($Bessel$)$
     $(C_q)$-system in $L^2(\R)$ while
    both integrals in $\eqref{e:bl cond r-s}$ converge.
\end{enumerate}
For $s= \infty$, the condition that the right-hand integral in
\eqref{e:bl cond r-s} is convergent should be replaced by the
condition that $g$ has compact support.
\end{theorem}
\begin{figure}
    \includegraphics{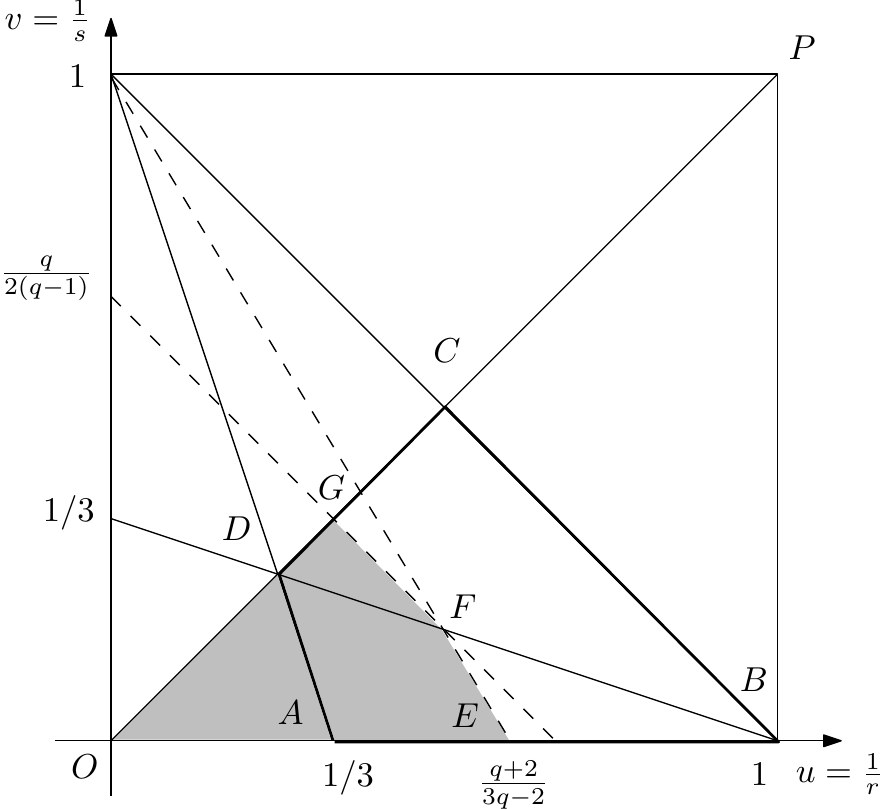}
    \caption{Illustration of Theorem \ref{main theorem} over $Q= [0,1]^2$.}
    \label{figure}
\end{figure}
We take a closer look at Figure 1.
The symmetric case of Theorem \ref{symmetric theorem} is precisely
Theorem \ref{main theorem} when restricted to the segment $CD$.
In particular, the extremal cases of Riesz bases ($q=2$) and exact
systems ($q=\infty$) correspond to the vertices $C = (1/2,1/2)$
and $D=(1/4,1/4)$ of the segment, respectively.

The non-symmetric Balian--Low type theorems for Riesz bases
and exact systems, mentioned in Section \ref{balian-low type section},  correspond to the segments BC and AD, respectively. 
Indeed, if the point $(1/r,1/s)$
is below the segment $BC$, i.e., the curve $\Gamma_2 = \{u+v=1\}$, then
it satisfies the condition in \eqref{e:bl cond for rb, c-d} and
therefore, for a generator of a Riesz basis, at least one of the
integrals in \eqref{e:bl cond r-s}  must diverge. Similarly, if the
point is below $AD$, i.e., the curve $\Gamma_{\infty} = \{3u+v=1\}$,
then the condition in \eqref{e:bl cond for exact, c-d} is
satisfied, which guarantees that for a generator of an exact system, at least one
of these integrals must diverge.

Theorem \ref{main
theorem} addresses generators of exact systems which cannot give
Riesz bases, and therefore is most interesting for the region
$ABCD$: the area between the curves $\Gamma_{2}$ and
$\Gamma_{\infty}$. As $q$ varies from $2$ to $\infty$, the curves
$\Gamma_q$ provide a continuous interpolation between $\Gamma_2$
and $\Gamma_{\infty}$, covering all of this area. Therefore, the
conditions described in Theorem \ref{main theorem} provide a
continuous transition between the Balian--Low type
conditions for generators of Riesz bases and the corresponding
conditions for the generators of exact systems.

Note that Theorem \ref{main theorem} does not address points
$(1/r,1/s)$ which are on the curve $\Gamma_q$. See also Remark
\ref{last remark} in this context.

\subsection{The structure of the paper}
We begin   by laying the 
groundwork for our proof of Theorem \ref{main result}. In Section \ref{reformulation section}, we
  relate    the $(C_q)$ property of Gabor systems   
with the regularity of the Zak transform of the generators.
It is known that  the time--frequency localisation of a function and the regularity of its Zak transform are connected. In 
Section \ref{smoothness section}, we study this connection further, obtaining both Lipschitz and integral type estimates.

With this, we give a proof for part $(a)$ of Theorem \ref{main result} in
Section \ref{proof - part a}.
To prove part $(b)$ of  the theorem, we introduce in Section \ref{building blocks section} the building blocks for the constructions needed, before completing the proof in Section \ref{final section}.
In Section \ref{remarks section} we give concluding remarks.

%
%
\section{A reformulation of the problem} \label{reformulation section}
We establish some machinery, formulated in Lemma \ref{l:basic lemma, function version}, which helps us determine whether a system $G(g,1,1)$ is a $(C_q)$-system  by looking at the Zak transform of its generator.

\subsection{Some notation}
For $d \in \N$, the Fourier transform of a function $g\in L^2(\R^d)$ is denoted by
$\hat{g}$ and defined as the usual extension of the Fourier
transform on $L^1(\R^d)$:
\begin{equation*}
    \hat{g}(\xi)= \int_{\R^d}g(t)\e^{-2\pi \im t\cdot \xi} \dif t,  \qquad \qquad \xi \in \R^d.
\end{equation*}
We set $Q=[0,1]^2$. The Fourier coefficients
of a
function $g \in L^1(Q)$ are given by
\begin{equation*}
     \hat{g}(m,n) = \iint_Q g(x,y) \e^{-2\pi \im (mx+ny)} \dif x \dif y, \qquad \qquad (m,n)\in \Z^2.
\end{equation*}
Whenever an $L^p$ integrable function is almost everywhere equal to a continuous function, we assume that
they are equal everywhere. This is possible since the pointwise estimates we make are only
used in integral expressions.

For functions defined on some subset of $\Omega \subset \R^d$, we use the notation $f_x^{(k)}$ for   the $k$-th partial derivative with respect to a coordinate $x$. By $C^k(\Omega)$, we denote the class of functions whose partial derivatives of order $k$ exist and are continuous on $\Omega$. The functions that satisfy this for every $k \in \N$, is said to be of the class $C^\infty(\Omega)$.
Also, by $C$ we denote constants which may change from step to step.

\subsection{A characterisation of exact $(C_q)$-systems}
A complete system $\{f_n\}$ in a Hilbert space $H$ is called
$\textit{exact}$ if it becomes incomplete when any one of its
members is removed. This condition holds if and only if there
exists a unique system $\{g_n\}\subset H$ such that $\langle
f_m,g_n\rangle = \delta_{m,n}$, where $\delta_{m,n}$ is
the Kroenecker delta.
In this case, $\{g_n\}$ is called the {dual} system of $\{f_n\}$.

The following characterization of exact $(C_q)$-systems can be
found in \cite{nitzan09}. We include a proof for the sake of
completeness. Note that if $q=2$, and the system is in addition a Bessel
system, then condition $(c)$ of this theorem coincides with the known
characterization of Riesz bases (see, for example,
\cite{young01}).

\begin{theorem}
\label{t:characterisation for exact c_q}
 Fix $q\geq 2$ and let
$\{f_n\}$ be a system in $H$. The following are equivalent.
\begin{enumerate}[label=$(\alph*)$]
\item The system $\{f_n\}$ is an exact $(C_q)$-system.
\item The system $\{f_n\}$ is exact and
    \begin{equation*}
        \Big(\sum|\langle f,g_n\rangle|^q\Big)^{\frac{1}{q}}\leq
        C\|f\|, \qquad \forall f \in H,
    \end{equation*}
   where $\{g_n\}$  is the dual
    system of $\{f_n\}$.
\item The system $\{f_n\}$ is complete and
    \begin{equation*}
        \Big(\sum |a_n|^q \Big)^{\frac{1}{q}}\leq C\Big\|\sum a_nf_n\Big\|,
    \end{equation*}
    for every finite sequence of numbers $\{a_n\}$.
\end{enumerate}
\end{theorem}

\begin{proof}
$(a)\Rightarrow (b)$: Let $\{g_n\}$ be the dual
system of $\{f_n\}$ and choose $f\in H$. Fix an integer $M>0$.
Since $\{f_n\}$ is a $(C_q)$-system,  there exists a finite linear
combination  ${\tilde{f}=\sum a_nf_n}$ that approximates $f$ in
norm, and satisfies
\begin{equation*}
    \Big(\sum^M_{n=1} |\langle \tilde{f}, g_n\rangle|^q\Big)^{\frac{1}{q}}
    =
    \Big(\sum^M_{n=1} |a_n|^q\Big)^{\frac{1}{q}}
    \leq C \norm{f}.
\end{equation*}
Since $f$ is approximated by $\tilde{f}$, we have  
\begin{equation*}
    \Big(\sum^M_{n=1} |\langle f, g_n\rangle|^q\Big)^{\frac{1}{q}}
 \leq C \norm{f}.
\end{equation*}
The conclusion follows.

 $(b)\Rightarrow (c)$: This implication is obvious.

 $(c)\Rightarrow (a)$: First, if $\{f_n\}$ is not exact then there exists an
$n_0$ for
 which $f_{n_0}$ lies in the closed span of $\{f_n\}_{n\neq n_0}$.
 So, for $\epsilon >0$, there exists a finite linear
 combination $\tilde{f}=\sum_{n\neq n_0}a_nf_n$ such that
 $\|f_{n_0}-\tilde{f}\|<\epsilon$. This implies that
 $(1+\sum_{n\neq n_0}|a_n|^q)^{{1}/{q}}\leq C\epsilon$. By choosing $\epsilon$
 sufficiently small, we get a contradiction.

 Next, let $f\in H$ and $\epsilon>0$. Since $\{f_n\}$ is complete,
 there exists a finite linear combination $\tilde{f}=\sum a_nf_n$ which
 approximates $f$ in norm and satisfies $\|\tilde{f}\|\leq\|f\|$.
 It follows that
$\left(\sum|a_n|^q\right)^{{1}/{q}}\leq C\|\tilde{f}\|\leq C\|f\|$, 
 and the proof is complete.
\end{proof}

\begin{remark} 
In particular, Theorem \ref{t:characterisation for exact c_q} implies that if a system $G(g,1,1)$ is
exact, then it is also a $(C_{\infty})$-system. This follows from
the implication $(b) \Rightarrow (a)$ and the fact that for such a
system, the dual system also takes the form $G(h,1,1)$ for some
function $h\in L^2(\R)$.
\end{remark}

\subsection{Exponential $(C_q)$-systems and weighted $L^2$ spaces}

Given a weight $w \in L^1(Q)$ satisfying $w > 0$ almost everywhere,
the weighted space $L^2_w(Q)$ is defined by
\begin{equation*}
    L^2_w(Q):= \left\{g : \norm{g}^2_{L^2_w(Q)} = \iint_Q|g|^2 \, w \; \dif x \dif y <\infty \right\}.
\end{equation*}
The system of exponentials
\begin{equation}\label{E}
    E :=\left\{\e^{2\pi \im (mx+ny)}\right\}_{m,n \in \Z}
\end{equation}
is complete in $ L^2_w(Q)$.
Moreover, it is easy to check that $E$ is exact in the space if
and only if $1/w \in L^1(Q)$. In this case, the dual system of $E$
consists of the functions 
\begin{equation}
\label{e:bi-orth}
h_{m,n} := \frac{1}{w} \, \e^{2\pi \im (mx + ny)}.
\end{equation}

\begin{lemma}
\label{l:basic lemma, weight version}
    Fix $q>2$ and let $w \in L^1(Q)$ satisfy $w > 0$ almost everywhere.
    \begin{enumerate}[label=$(\alph*)$]
        \item If $1/w \in L^{\frac{q}{q-2}}(Q)$, then
        $E$ is an exact $(C_q)$-system in $L^2_w(Q)$.
        \item If  there exists a
        function $g\in L^2_w(Q) \cap L^1(Q)$ such that
        \begin{equation*}
            \sum_{m,n\in \Z}|\hat{g}(m,n)|^q=\infty,
        \end{equation*}
        then $E$ is not an exact $(C_q)$-system in $L^2_w(Q)$.
    \end{enumerate}
\end{lemma}
\begin{proof}
    To see $(a)$, we use condition $(b)$ of Theorem \ref{t:characterisation for exact c_q}.  
    Indeed, we first note that  if $1/w \in L^{\frac{q}{q-2}}(Q)$, then $1/w \in
    L^1(Q)$. Therefore,  $E$ is exact in the space and the dual system is given by \eqref{e:bi-orth}.
    For $g\in L^2_w(Q)$, we evaluate
    \begin{align*}
        \sum_{m,n \in \Z} |\langle g, h_{m,n}\rangle_{L^2_w(Q)}|^q
         &=
        \sum_{m,n \in \Z} \Big|\iint_Q g \; \frac{1}{w} \, \e^{-2\pi \im (mx + ny)}w \; \dif x \dif y \Big|^q
        =
        \sum_{m,n \in \Z} |\hat{g}(m,n)|^q.
    \end{align*}
    By the Hausdorff--Young inequality, the last expression is smaller
    than $\|g\|_{L^p(Q)}^q$, where ${1}/{p}+{1}/{q}=1$. We
    can now use H{\" o}lder's inequality to check that
    \begin{equation*}
        \|g\|_{L^p(Q)}=\Big\|g\sqrt{w}\cdot \frac{1}{\sqrt{w}} \Big\|_{L^p(Q)}
        \leq
        \|g\|_{L^2_w(Q)} \Big\|\frac{1}{w}\Big\|^{\frac{1}{2}}_{L^{\frac{q}{q-2}}(Q)},
    \end{equation*}
    and $(a)$ follows.

    A similar argument can be used to prove $(b)$.
\end{proof}

\subsection{The Zak transform and Gabor $(C_q)$-systems}

The following definition is commonly used in the study of lattice
Gabor-systems (see, for example, \cite{heil07}).
\begin{definition}
\label{d:zac transform}Let $g \in L^2(\R)$. The Zak transform of
$g$ is given by 
\begin{equation*}
     Zg (x,y)= \sum_{k\in \Z} g(x-k) \e^{2\pi i k y}, \quad \quad \forall (x,y)\in \R^2.
\end{equation*}
\end{definition}
One can easily verify that, for every $g\in L^2(\R)$, the function
$Zg$ is {quasi-periodic} on $\R^2$. That is, for every
$(x,y)\in \R^2$, it satisfies
\begin{equation}
Zg(x,y+1)=Zg(x,y) \quad \text{and} \quad Zg(x+1,y)=e^{2\pi \im
y}Zg(x,y).
\end{equation}
This implies that $Zg$ is determined uniquely by its values on
$Q$. It is well-known that when restricted to $Q$, the Zak
transform induces   a unitary operator from $L^2(\R)$ onto $L^2(Q)$. In
particular, this means that any quasi-periodic function, for which the
restriction to $Q$ is square integrable, is the image under the
Zak transform of some function $g\in L^2(\R)$. Throughout this
paper $Zg$ denotes either the Zak transform of $g$ or its
restriction to $Q$. The use of this notation will be clear from
the context.

We now explain how   the weighted spaces $L^2_{\abs{Zg}^2}(Q)$ can be used to study Gabor systems $G(g,1,1)$.
First note that 
\begin{equation}
\label{zak of gg11}
    Z\left\{g(t-n)\e^{2\pi i m t} \right\}(x,y) = \e^{2\pi \im (mx - ny)}
    Zg(x,y).
\end{equation}
Therefore, the system $G(g,1,1)$ is complete in $L^2(\R)$ if
and only if $Zg \neq 0$ almost everywhere. Next,
let $g\in L^2(\R)$ be such that $Zg \neq 0$ almost everywhere and
denote by $U_g : L^2(\R) \rightarrow L^2_{|{Zg}|^2}(Q)$ the
operator
\begin{equation*}
     U_g : h \longmapsto \frac{Zh}{Zg}.
\end{equation*}
It is clear that $U_g$ is a unitary bijection, and it follows from
\eqref{zak of gg11} that the image of the system $G(g,1,1)$ under
this operator is the system $E$ defined in \eqref{E}. Hence,
$G(g,1,1)$ is an exact system, a $(C_q)$-system, or a frame in
$L^2(\R)$ if and only if the same can be said about the system $E$
in $L^2_{|{Zg}|^2}(Q)$.

The following reformulation of Lemma \ref{l:basic lemma, weight
version} is now immediate.
\begin{lemma}
\label{l:basic lemma, function version}
    Fix $q>2$ and let $g \in L^2(\R)$ satisfy $Z g\neq 0$ almost everywhere.
    \begin{enumerate}[label=$(\alph*)$]
        \item If ${1}/{|Z g|^2} \in L^{\frac{q}{q-2}}(Q)$, then
        $G(g,1,1)$ is an exact $(C_q)$-system in $L^2(\R)$.
        \item If  there exists a
        function $f\in L^2_{|Z g|^2}(Q)\cap L^1(Q)$ with
        \begin{equation*}
            \sum_{m,n \in \Z}|\hat{f}(m,n)|^q=\infty,
        \end{equation*}
        then $G(g,1,1)$ is not an exact $(C_q)$-system in $L^2(\R)$.
    \end{enumerate}
\end{lemma}
\begin{remark} \label{bessel remark}
Similarly, one can show that a system $G(g,1,1)$ is a Bessel
system if and only if the weight $|Zg|^2$ is bounded from
above almost everywhere.  As follows
from Lemma 6 below, this condition holds for all cases we
discuss in the context of Theorem \ref{main result}, i.e., whenever the point $(1/r, 1/s)$ is below the curve
$BC=\Gamma_2$ (see \eqref{gamma q} and Figure \ref{figure}).
\end{remark}

\section{Smoothness properties of the zak transform} \label{smoothness section}

In this section, we study the connection between the time--frequency localisation
of a function and the regularity of its Zak transform. This is done
both in terms of certain integral estimates as well as pointwise Lipschitz type estimates.
\subsection{Smoothness and the Fourier transform}
For $h\in \R$ and $k \in \N$, the operator $\tau_h^k : L^2(\R)
\rightarrow L^2(\R)$ is defined by
\begin{equation*}
     \tau_h g (t) = g(t+h) - g(t) \quad \text{and} \quad \tau_h^k g = \tau_h \tau_h^{k-1} g.
\end{equation*}
We use the convention $\tau_h^0 g = g$.
Since $\widehat{\tau_h g}(\xi) = (\e^{2\pi \im \xi
h}-1)\widehat{g}(\xi)$, it follows by induction that
\begin{equation} \label{part2_transform of the tau operator}
     \widehat{\tau_h^k g}(\xi) = (\e^{2\pi \im \xi h} - 1)^k \widehat{g}(\xi).
\end{equation}
One can now easily deduce the following classical relation (see
\cite[p. 139--140]{stein70} for the cases $k=1,2$), which connects
the smoothness of a function in $L^2(\R)$ to the decay of its
Fourier transform:
For $0 < r < 2k$,  there exists a constant $C>0$ such that
\begin{equation} \label{part2_steins equality}
     \iint_{\R^2} \frac{\abs{\tau_h^k g(t)}^2}{\abs{h}^{1+r}} \dif t \dif h
     =
     C \int_\R \abs{\xi}^r \abs{\hat{g}(\xi)}^2 \dif \xi.
\end{equation}
Indeed, by Parseval's identity and the equation \eqref{part2_transform
of the tau operator},
\begin{equation*}
     \int_\R \abs{\tau_h^k g(t)}^2 \dif t
     =
     \int_\R \abs{\widehat{\tau_h^kg} (\xi)}^2 \dif \xi
     =
     \int_\R \abs{\e^{2\pi \im \xi h} - 1}^{2k} \abs{\widehat{g}(\xi)}^2 \dif \xi.
\end{equation*}
Whence, by a an appropriate change of variables,
\begin{equation*}
     \iint_{\R^2} \frac{\abs{\tau_h^kg(t)}^2}{\abs{h}^{1+r}} \dif t \dif h
     =
     \int_\R \frac{\abs{\e^{2\pi \im h } - 1}^{2k}}{\abs{h}^{1+r}} \dif h
     \cdot
     \int_\R \abs{\xi}^r \abs{\hat{g}(\xi)}^2 \dif \xi.
\end{equation*}
The following two
lemmas list some basic properties of the operator $\tau_h$ which
are used in later sections. 
\begin{lemma} \label{part2_basic basic lemma}
    For any functions $f$ and $g$ on $\R$, the following relations hold.
        \begin{align*}
             (a)&\quad \abs{\tau_h^k g(t)} \leq 2^k \sum_{j=0}^k \abs{g(t+jh)}.\\
           (b)&\quad \tau_h^k (fg) (t) = \sum_{j=0}^k \binom{k}{j} \tau_h^j f(t) \tau_h^{k-j} g(t + jh).
        \end{align*}
        Moreover, if $h\geq 0$ and $g \in C^k[t,t+kh]$, then
        \begin{equation*}
            \hspace{-1.9cm} (c)\quad \abs{\tau_h^k g(t)} \leq \abs{h}^k \sup_{\xi \in [t,t+kh]} \abs{g^{(k)}(\xi)}.
        \end{equation*}
     For $h<0$, the same estimate holds over the interval $[t+kh,t]$.
\end{lemma}
This lemma can be proved easily using an inductive process and the mean value theorem (for estimate $(c)$). We leave the details
to the reader.%

\begin{lemma} \label{bonus lemma}
Fix $k \in \N$ and $0<r<2k$.
Suppose $U \subset \R$ and let $g$ be a function on $\R$.
\begin{enumerate}[label=$(\alph*)$]
\item If $\int_{U}\abs{g(t+\eta)}^2 \dif t$ is bounded uniformly for
all $\eta\in \R$, then
\begin{gather*}
     \int_\R \int_{U} \frac{\abs{\tau_h^k g(t)}^2}{h^{1+r}} \dif t  \dif h < \infty
    \; \; \iff \; \;
    \int_{-1}^1 \int_{U}   \frac{\abs{\tau_h^k  g(t)}^2}{h^{1+r}} \dif t \dif h < \infty.
\end{gather*}

\item Suppose that $U$ is bounded. If g is locally square integrable and $\phi \in C^k(\R)$, then
\begin{gather*}
    \int_{-1}^1 \int_{U} \frac{\abs{\tau_h^k g(t)}^2}{h^{1+r}} \dif t\dif h < \infty
     \; \implies \;
    \int_{-1}^1 \int_{U}   \frac{\abs{\tau_h^k (\phi g)(t)}^2}{h^{1+r}} \dif t \dif h < \infty.
\end{gather*}
\end{enumerate}
\end{lemma}
\begin{proof}
Throughout this proof we use the notation $g(t+\eta) =
g_\eta(t)$.

$(a):$ As follows from Lemma
\ref{part2_basic basic lemma}$(a)$, if $g$ satisfies the conditions
above then
\begin{equation*}
    \int_{\abs{h}>1} \int_{U}  \frac{\abs{\tau_h^k g(t)}^2}{h^{1+r}} \dif t  \dif h < \infty.
\end{equation*}
The conclusion follows.

$(b):$ We prove this
by induction on $k \in \N$. For $k=1$, it is straight-forward since
by  Lemma \ref{part2_basic basic lemma}$(b)$ we have
\begin{equation}\label{shahaf equation} 
    \tau_h (u v) = \tau_h u \cdot v + u_h \cdot \tau_h v.
\end{equation}
Indeed, this identity applied to $u= \phi$ and $v= g$ yields the inequality
\begin{multline*}
    \frac{1}{2}\int_{-1}^1 \int_{U}    \frac{\abs{\tau_h (\phi g)(t)}^2}{h^{1+r}} \dif t\dif h
    \leq
    \int_{-1}^1 \int_{U}   \frac{\abs{(\tau_h \phi \cdot g)(t)}^2}{h^{1+r}} \dif t \dif h
    \\+
    \int_{-1}^1 \int_{U}   \frac{\abs{(\phi_h \cdot \tau_hg)(t)}^2}{h^{1+r}} \dif t \dif h.
\end{multline*}
The second term on the
right-hand side is finite since $\phi$ is   bounded on any
compact set. To see that the first term is finite, apply 
Lemma \ref{part2_basic basic lemma}$(c)$ to $\phi$, and conclude
using the facts that $g$ is locally square integrable and $r < 2$.

Next, assume that $(b)$
holds for $k < n$ and that
\begin{equation}
\label{cond in induction} \int_{-1}^1 \int_{U}
\frac{\abs{\tau_h^n g(t)}^2}{h^{1+r}} \dif t\dif h < \infty.
\end{equation}

By \eqref{shahaf equation}, in the same way as above, we have  
\begin{equation*}
    \tau_h^n (\phi g) = \tau_h^{n-1} \tau_h (\phi \cdot g) = \underbrace{\tau_h^{n-1} (\tau_h \phi \cdot g)}_{A_1} + \underbrace{\tau_h^{n-1}(\phi_h \cdot \tau_h g)}_{B_1}.
\end{equation*}
By the induction hypothesis,
$B_1$ gives rise to a finite term in the corresponding integral.
Indeed, note that $\tau_h^n g=\tau_h^{n-1} \tau_h g$ and apply the induction
hypothesis to \eqref{cond in induction} with $k=n-1$. On the other
hand, again by \eqref{shahaf equation}, now applied with $u=\tau_h \phi$ and $v=g$, we have 
\begin{equation*}
    A_1 = \tau_h^{n-2} \tau_h (\tau_h \phi \cdot g) =  \underbrace{ \tau_h^{n-2} (\tau_h^2 \phi \cdot g) }_{A_2} + \underbrace{\tau_h^{n-2} (\tau_h \phi_h \cdot \tau_h g )}_{B_2}.
\end{equation*}
We apply a version of the relation \eqref{shahaf
equation} (replace $v$ by $\tau_h v$): 
\begin{equation} \label{repeated use guy}
    \tau_h u \cdot \tau_h v =  \tau_h (u \cdot \tau_h v) -  u_h \cdot \tau_h^2 v,
\end{equation}
and find that
\begin{equation*}
    B_2 
    =  \tau_h^{n-1}( \phi_{h} \cdot \tau_h g)  - \tau_h^{n-2} (\phi_{2h} \cdot \tau_h^2 g) .
\end{equation*}
As above, $B_2$ gives rise to
finite terms in the corresponding integral expressions. Indeed, apply the induction hypothesis to \eqref{cond in induction}
with $k=n-1$ and $k=n-2$ (use $\tau_h g$ and $\tau_h^2 g$ in place
of $g$, respectively).

We iterate this process for $m\leq n$, applying \eqref{repeated use guy} repeatedly in each step, to get:
\begin{align*}
    A_m   =   \tau_h^{n-m} (\tau_h^m \phi \cdot g ) \qquad \text{and} \qquad  B_m   &=  \tau_h^{n-m} (\tau_h^{m-1} \phi_h \cdot \tau_h g) \\
    &=  \sum_{j=1}^{m}(-1)^{j-1} \binom{m-1}{j-1}  \tau_h^{n-j} (\phi_{jh} \cdot \tau_h^{j} g).
\end{align*}
In each step, due to the induction
hypothesis, the terms generated by the $B_m$ yield corresponding
finite integrals. The final term $A_n = \tau_h^n \phi \cdot g$
gives a convergent integral by Lemma \ref{part2_basic basic lemma}$(c)$.
\end{proof}

\subsection{Smoothness and the Zak transform} \label{smoothness and zak subsection}
For functions on $\R^2$, we define the operators $\Delta_h^k$ and
$\Gamma_h^k$, as analogues to the operator $\tau_h^k$, by
\begin{align*}
        \Delta_hF(x,y) &:= F(x+h,y) - F(x,y),  \\
        \Gamma_hF(x,y) &:= F(x,y+h) - F(x,y),
\end{align*}
and the relations
\begin{equation*}
     \Delta_h^kF = \Delta_h \Delta_h^{k-1} F \quad \text{and} \quad \Gamma_h^kF = \Gamma_h \Gamma_h^{k-1} F.
\end{equation*}
As above, we use the convention $\Delta_h^0 F = \Gamma_h^0F = F$.
By setting $F_y(x) = F(x,y)$, we can write $\Delta_h^k F(x,y) = \tau_h^k F_y(x)$, and similarly for $\Gamma_h^k$. This allows us
to carry over results on $\tau_h^k$ to the operators $\Delta_h^k$ and $\Gamma_h^k$.

 For the operator $\Delta_h^k$, the identity in
\eqref{part2_steins equality} takes the following form:
For $0 < r < 2k$,
there exists a constant $C>0$ such that for every $F\in L^2(\R^2)$
we have
\begin{equation} \label{part2_steins equality for 2 variebals}
     \iiint_{\R^3} \frac{\abs{\Delta_h^k F(x,y)}^2}{\abs{h}^{1+r}} \dif x \dif y \dif h
     =
     C \iint_{\R^2} \abs{u}^r \abs{\hat{F}(u, v)}^2 \dif u \dif v.
\end{equation}
A similar identity holds for the operator $\Gamma_h^k$.
\begin{remark} \label{part2_delta and gamma remark} 
    	Lemma \ref{part2_basic basic lemma} remains true in the two variable case when $\tau_h$ is replaced by either
    	$\Delta_h$ or $\Gamma_h$, and the appropriate modifications are made.
	In what follows, these properties will be referred to as
	remarks \ref{part2_delta and gamma remark}$(a)$, \ref{part2_delta
	and gamma remark}$(b)$ and \ref{part2_delta and gamma
	remark}$(c)$, respectively.
\end{remark}
\begin{remark} \label{bonus remark}
	The properties listed in Lemma \ref{bonus lemma} also hold,
	under appropriate modifications, for the operators $\Delta_h$ and
	$\Gamma_h$. 
    In what follows, these properties will be referred to as
remarks \ref{bonus remark}$(a)$ and \ref{bonus remark}$(b)$, respectively.
\end{remark}
A connection between  the time--frequency localisation of a function $g$ and the smoothness of its Zak transform is now given in the following lemma.
We note that the implications $(i)\Rightarrow(iii)$ and $(iv)\Rightarrow(vi)$, were first proved in \cite{gautam08}.
\begin{lemma} \label{part2_stein lemma for delta and gamma}Let $k \in \N$ and $0< r,s <
2k$. For every $g \in L^2(\R)$ we have: 
\begin{enumerate}[label=$(\alph*)$]
\item The following conditions are equivalent.
    \begin{align*}
        (i) & \int_\R \abs{\xi}^r \abs{\hat{g}(\xi)}^2 \dif \xi<\infty. \\
        (ii) & \int_\R \iint_{[0,1]^2} \frac{\abs{\Delta_h^k Zg(x,y)}^2}{\abs{h}^{1 + r}} \dif x \dif y \dif h<\infty. \\
        (iii) & \; \; \textrm{For every compactly supported function } \psi \in
        C^{k}(\R) \\
        &\iint_{\R^2} \abs{u}^r \abs{\widehat{\psi Zg}(u, v)}^2 \dif u \dif
        v<\infty.
    \end{align*}
\item Similarly, the following conditions are equivalent.
\begin{align*}
	(iv)&\int_\R \abs{t}^s \abs{g(t)}^2 \dif t  < \infty.\\
 	(v) & \int_\R \iint_{[0,1]^2} \frac{\abs{\Gamma_h^k Zg(x,y)}^2}{\abs{h}^{1+s}} \dif x \dif y \dif h < \infty.\\ 
	(vi) & \;\; \textrm{For every compactly supported function } \psi \in
        C^{k}(\R) \\
        &\iint_{\R^2} \abs{v}^s \abs{\widehat{\psi Zg}(u, v)}^2 \dif u \dif
        v<\infty.
 \end{align*}
\end{enumerate}
\end{lemma}
\begin{proof}
    $(i)\Leftrightarrow (ii):$ In fact, an even stronger result
    holds: there exists a constant $C>0$ such that for $g \in
    L^2(\R)$ we have
    \begin{equation}\label{zak and smoothness, identity}\int_\R \iint_{[0,1]^2} \frac{\abs{\Delta_h^k Zg(x,y)}^2}{\abs{h}^{1 + r}} \dif x \dif y \dif h
        \quad = \quad
        C  \int_\R \abs{\xi}^r \abs{\hat{g}(\xi)}^2 \dif \xi.
        \end{equation}
To see this, first note that $\Delta_h Z g= Z \tau_h g$. So, by
induction $\Delta_h^k Z g = Z \tau_h^k g$.
    Since the Zak transform is a unitary operator, this implies that
    \begin{equation*}
        \iint_{[0,1]^2} \abs{\Delta_h^k Zg }^2 \dif x \dif y
        =
        \int_\R \abs{\tau_h^k g(t)}^2 \dif t.
    \end{equation*}
    Hence, \eqref{zak and smoothness, identity} follows from \eqref{part2_steins equality}.

$(ii)\Leftrightarrow (iii):$ As follows from the identity
\eqref{part2_steins equality for 2 variebals} and Remark
\ref{bonus remark}$(a)$, it is enough to show that $(ii)$ holds if
and only if 
\begin{equation}
\label{what we need to show in order to show guatam}
 \int_{-1}^1\iint_{\R^2} \frac{\abs{\Delta_h^k (\psi Zg)(x,y)}^2}{\abs{h}^{1+r}} \dif x \dif y \dif
 h<\infty
 \end{equation}
for every compactly supported function $\psi \in
C^{\infty}(\R^2)$.

Assume first that \eqref{what we need to show in order to show
guatam} is satisfied for some function $g\in L^2(\R)$. Let $\psi
\in C^{\infty}(\R^2)$ be a compactly supported function which
satisfies $\psi=1$ on $[-k,k+1] \times [0,1]$. 
Note that $\Delta_h^k( \psi Zg) = \Delta_h^k Zg$ for $(x,y) \in Q$ and $h \in [-1,1]$.
So the integral in $(ii)$ can be written as
\[ \int_{-1}^1 \iint_{[0,1]^2} \frac{\abs{\Delta_h^k (\psi
Zg)(x,y)}^2}{\abs{h}^{1 + r}} \dif x \dif y \dif h
+\int_{|h|>1}
\iint_{[0,1]^2} \frac{\abs{\Delta_h^k Zg(x,y)}^2}{\abs{h}^{1 + r}}
\dif x \dif y \dif h.
\]
The first integral in this sum converges by \eqref{what we need to show in order to show guatam}, while
the second integral converges by an application of Remark \ref{part2_delta and gamma remark}$(a)$ and the quasi-periodicity
of $Zg(x,y)$.

Next, suppose $(ii)$ holds for some $g\in L^2(\R)$ and let $\psi
\in C^{\infty}(\R^2)$ be a compactly supported function. It
follows by the quasi-periodicity of $Zg(x,y)$
that for any positive integer $n$,
\begin{equation} \label{emergency}
\int_{-1}^1 \iint_{[-n,n]^2} \frac{\abs{\Delta_h^k
Zg(x,y)}^2}{\abs{h}^{1 + r}} \dif x \dif y \dif h<\infty.
\end{equation}
Choose $n \in \N$ big enough for
the support of $\psi$ to be included in $[-n+k,n-k]  \times [-n,n]$. This allows
us to write \eqref{what we need to show in order to show guatam}
as
\begin{equation*}
    \int_{-1}^1 \iint_{[-n,n]^2} \frac{\abs{\Delta_h^{k} (\psi Zg)(x,y) }^2}{\abs{h}^{1+r}}  \dif x \dif y \dif h.
\end{equation*}
By Remark \ref{bonus remark}$(b)$, the inequality \eqref{emergency} implies
that this is finite.

    $(iv)\Leftrightarrow (v):$ Let $S_h: L^2(\R)\rightarrow L^2(\R)$ be the operator defined by
    \begin{equation*}
        \left(S_hf\right) (t) = f(t) (\e^{-2\pi \im n h} - 1 ), \quad \text{for} \quad n \leq t < n+1.
    \end{equation*}
    It is easily verified that $\Gamma_h Z g = Z S_h g$ on $Q$.
    So, by induction $\Gamma_h^k  Z g = Z  S_h^k g$. Again, since the Zak transform
    is a unitary operator, we get
    \begin{equation*}
        \iint_{[0,1]^2} \abs{\Gamma_h^k Zg}^2 \dif x \dif y
        =
        \int_\R \abs{S_h^k g(t)}^2 \dif t.
    \end{equation*}
    As above,
    \begin{align*}
        \iint_{\R^2} \frac{ \abs{S_h^k g(t)}^2 }{\abs{h}^{1+s}} \dif t \dif h
        &=\sum_{n\in \Z} \int_\R \frac{\abs{\e^{-2\pi \im n h} - 1}^{2k}}{\abs{h}^{1+s}} \dif h
        \int_{n}^{n+1} \abs{g(t)}^2 \dif t \\
         &=
        C \sum_{n \in \Z} \abs{n}^s \int_{n}^{n+1} \abs{g(t)}^2 \dif t.
    \end{align*}
    It is clear that the right-hand side converges if and only if the same is true for the integral
$\int_\R \abs{t}^s \abs{g(t)}^2 \dif t$.

$(v)\Leftrightarrow (vi):$ This is proved in a similar way as $(ii)\Leftrightarrow (iii)$.
\end{proof}

\subsection{Lipschitz type conditions and the Zak transform}

The following result  appears implicitly in   \cite[Theorem 3.2]{heil_powell06}.
\begin{lemma} \label{prelim: zero of zac}
   Let $g\in L^2(\R)$ and $r,s > 0$ be such that ${1}/{r}+{1}/{s}<1$.
   If both integrals in \eqref{e:bl cond r-s}
are finite,
    then $Zg$ is continuous on $\R^2$ and has a zero in $Q$.
\end{lemma}
\begin{proof}
    Since this result will be of importance in what follows, we give a short indication of a  proof.
    By combining Proposition 1 and Theorem 1 of \cite{grochenig96},
    one can check that
    a function for which both integrals in \eqref{e:bl cond r-s} are finite 
    satisfies $\sum_{k \in \Z}
    \norm{g}_{L^\infty(k,k+1)}< \infty$ (i.e., it belongs to the 
    Wiener space). Since, in addition, $\hat{g}\in L^1$ implies that $g$ is continuous,
    it follows that $Zg$ is also continuous and therefore has a zero in $Q$ (see Lemma 8.2.1 part (c) and  Lemma 8.4.2 in \cite{grochenig01}).
%
\end{proof}
 The next lemma
establishes Lipschitz type conditions for the Zak transform of  
functions satisfying the conditions of Lemma \ref{prelim: zero of zac}.
It is of particular interest for us that this
lemma describes how ``deep''  the zero of $Zg$ must be.
\begin{lemma} \label{zak lemma}
    Let $r>0$ and $s>0$ satisfy
    ${1}/{r}+{1}/{s}<1$, and define
     \begin{equation} \label{e:who is phi}
        \phi_{r,s}(x) =
        \left\{ \begin{array}{cc}
            \abs{x}^{  2 - r \left( \frac{3}{r} + \frac{1}{s} - 1 \right) }
            &
            \text{if} \quad   \frac{3}{r} + \frac{1}{s} > 1,
            \vspace{0.2cm}
            \\
            \abs{x}^2 \log \left( 1 +  \frac{1}{\abs{x}} \right)
            &
            \text{if} \quad    \frac{3}{r} + \frac{1}{s} = 1,
            \vspace{0.2cm}
            \\
            \abs{x}^2
            &
            \text{if} \quad    \frac{3}{r} + \frac{1}{s} < 1.
         \end{array} \right.
    \end{equation}
Suppose that for   $g \in L^2(\R)$ both integrals in \eqref{e:bl cond r-s}
are finite. Then given $(a,b) \in \R^2$ we have
    \begin{equation}\label{oh my lord}
        \abs{Zg(x,y) - Zg(a,b)}^2 \leq C \Big( \phi_{r,s}(x-a) + \phi_{s,r}(y-b) \Big)
    \end{equation}
on $\R^2$, where $C>0$ is a constant not
depending on $x$ and $y$.
\end{lemma}
As follows from the proof, and from the fact that $Zg$ is a
quasi-periodic function, the constant $C$ can be chosen in such a way that it does not depend on
the point $(a,b)$. This, however, is not needed for our purposes.

\begin{proof}
    By Lemma \ref{prelim: zero of zac}, $Zg$ is continuous, and so it is enough to prove \eqref{oh my lord} in a
     neighbourhood of $(a,b)$.
    Choose a compactly supported function $\psi \in C^\infty(\R^2)$
    that satisfies $\psi \equiv 1$ in a neighbourhood $U$ of $(a,b)$.
      By the Cauchy--Schwarz inequality, the following estimate holds for every
     $(x,y) \in U$,
    \begin{align*}
        \abs{Zg(x,y) - Zg(a,b)} &= \abs{(\psi Zg)(x,y) - (\psi Zg)(a,b)} \\
        &\leq \iint_{\R^2}  \Abs{1 - \e^{2\pi \im ((x-a)u + (y-b)v)}}\abs{\widehat{(\psi Zg)}(u,v) } \dif u \dif v \\
         &\leq 2 \bigg(\underbrace{\iint_{\R^2} (1 + \abs{u}^r+\abs{v}^s) \abs{\widehat{(\psi Zg)}(u, v)}^2 \dif u  \dif
        v}_{(I)}\bigg)^{1/2} \\
        &\qquad \qquad \times
        \bigg( \underbrace{\iint_{\R^2} \frac{\sin^2 \pi((x-a)u + (y-b)v)}{1 + \abs{u}^r + \abs{v}^s}
        \dif u \dif v}_{(II)} \bigg)^{1/2}.
    \end{align*}
    The integral $(I)$ is finite by Lemma \ref{part2_stein lemma for delta and gamma}. Indeed,
    the function
     $\psi Zg$ satisfies both condition $(iii)$ and $(vi)$ of the lemma.
    As for $(II)$, the symmetry of the integrand and the inequality $\sin^2(x+y) \leq 2(\sin^2 x + \sin^2 y)$ imply that
    \begin{equation*}
         (II) \leq
         8 \underbrace{\iint_{[0, \infty)^2}  \frac{\sin^2 \pi (x-a)u}{1 + u^r + v^s} \dif u \dif v}_{(III)}
         + 8 \underbrace{\iint_{[0, \infty)^2}  \frac{\sin^2 \pi (y-b)v}{1 + u^r + v^s} \dif u \dif v}_{(IV)}.
    \end{equation*}
    By an appropriate change of variables, we get
    \begin{equation*}
      (III) = \abs{x-a}^{r - 1 - \frac{r}{s}} \underbrace{\iint_{[0, \infty)^2} \frac{\sin^2 \pi u}{\abs{x-a}^r + u^r + v^s} \dif u \dif v}_{(\widetilde{III})}.
    \end{equation*}
    To estimate this integral we divide the area of integration into two parts:
    \begin{align*}
        Q &= [0,1]^2, &
        Q^C &= [0,\infty)^2 \setminus Q&.
    \end{align*}
    This induces the splitting $(\widetilde{III}) = (III_Q) + (III_{Q^c})$.
    We use the inequalities $\abs{\sin x} \leq x$ and 
    \begin{equation} \label{secret trick}
        c_\beta  (x + y)^\beta \leq x^\beta + y^\beta \leq C_\beta (x+y)^\beta, \qquad
        \forall \beta>0,\:\: x\geq 0,\:\: y \geq 0,
    \end{equation}
    to find that, since $s>1$,
    \begin{align*}
        (III_Q)  &\leq C \iint_Q  \frac{u^2}{(\abs{x-a}^{\frac{r}{s}} + u^{\frac{r}{s}} + v )^s} \dif u \dif v \\
                &\leq C \left( \int_0^1 \frac{u^2}{(\abs{x-a}^{\frac{r}{s}} + u^{\frac{r}{s}})^{s-1}} \dif u +  1 \right)  \\
           &\leq C\left( \int_0^1 \frac{u^2}{(\abs{x-a}^3 + u^3)^{(s-1)\frac{r}{3s}}} \dif u +1 \right)
            =
            C\left( \int_{\abs{x-a}^3}^{1 + \abs{x-a}^3} \frac{\dif w}{w^{(s-1)\frac{r}{3s}}} \dif u +1 \right).
    \end{align*}
    Hence,
    \begin{equation*}
        (III_Q) \leq
            \left\{
            \begin{array}{cc}   C &  \text{if} \quad \frac{3}{r} + \frac{1}{s} > 1, \vspace{0.2cm} \\
                                C  \log \left( 1 + \frac{1}{\abs{x-a}} \right) & \text{if} \quad  \frac{3}{r} + \frac{1}{s} = 1, \vspace{0.2cm} \\
                                C \abs{x-a}^{3 - r+ \frac{r}{s}} & \text{if} \quad
                                \frac{3}{r} + \frac{1}{s} < 1.
                \end{array}
            \right.
    \end{equation*}
    By similar estimates, and the fact that
    ${1}/{r}+{1}/{s}<1$, it is easy to check that $(III_{Q^c})\leq C$.

    Repeating these arguments for the integral $(IV)$, the lemma is established.
\end{proof}

\begin{remark} \label{phi remark}
Lemma \ref{zak lemma} holds also in the extremal case $s=\infty$,
i.e., when $g$ is compactly supported. This can be shown using
similar arguments. See also \cite{heil_powell09}.
\end{remark}

\begin{remark}  
Lemma \ref{zak lemma} is sharp. 
That is, for every $r,s > 0$ as above and $\eps>0$, there exist a function $g$ for which both
integrals in \eqref{e:bl cond r-s} converge and a point $(a,b)\in Q$ such that the
inequality
\begin{equation}
        \abs{Zg(x,y) - Zg(a,b)}^2 \geq C \Big( \phi_{r+\eps,s+\eps}(x-a) + \phi_{s+\eps,r+\eps}(y-b) \Big)
    \end{equation}
holds in a neighbourhood of $(a,b)$ (note that for $\epsilon >0$ we have $\phi_{r+\epsilon, s+\epsilon}(x) \leq C \phi_{r,s}(x)$ in $Q$).
 Indeed, the functions constructed
in the   proof of part $(b)$ of Theorem \ref{main result} provide the
required estimates.
\end{remark}


\section{Theorem \ref{main theorem} -- First part} \label{proof - part a}

    To prove Theorem \ref{main result}$(a)$ we will combine  
    Lemma \ref{l:basic lemma, function version}$(b)$ with lemmas \ref{prelim: zero of zac} and \ref{zak lemma}. In order to do so, we need to find a family of test
    functions for which we are able to estimate both   $L^2_{|Zg|^2}(Q)$ norms and the $\ell^q$ norm of their Fourier coefficients. 
    \subsection{A family of test functions}
    For $\a>0$ and $\b>0$, a suitable family of functions is given by
    \begin{equation}
		\label{e: def of f, a,b}
        f_{\a,\b}(x,y):=\frac{1}{\left[1+(1-|x-\frac{1}{2}|^{\a})\e^{2\pi \im y}\right]^{\b}}, \quad (x,y)\in Q.
    \end{equation}
    Here $z^{\b}=\e^{\b \log z}$, where $\log z$ is the
    principle value of the logarithm on $\C\setminus [-\infty,
    0]$.
    Note that the functions $f_{\a,\b}$ satisfy
    \begin{equation}
        \label{e:growth of f_a,b}
        |f_{\a,\b}(x,y)|^2\leq
        \frac{C}{\left(|x-\frac{1}{2}|^{2\a}+|y-\frac{1}{2}|^2\right)^{\b}}
    \end{equation}
    for some constant $C=C(\alpha,\beta)$.
    Indeed, for $(x,y)\in Q$ we have
    \begin{align*}
        \left|1+\left(1-\Big|x-\frac{1}{2}\Big|^{\a}\right)\e^{2\pi \im y}\right|^2
        &=
        \Big|x-\frac{1}{2}\Big|^{2\a}+2\left(1-\Big|x-\frac{1}{2}\Big|^{\a}\right)\left(\cos{2\pi y}+1\right)
        \\
        &\geq
        C \left(\Big|x-\frac{1}{2}\Big|^{2\a}+\Big|y-\frac{1}{2}\Big|^2\right).
    \end{align*}

    The following Lemma provides the required estimate for the
    Fourier coefficients of $f_{\a,\b}$.
    \begin{lemma} \label{coefficient lemma}
    Fix $q>2$. For every $0<\a<1$ and
    $(1-{1}/{q})(1+{1}/{\a})\leq \b<(1+{1}/{\a})$, the
    function $f=f_{\a,\b}$ belongs to $L^1(Q)$ and its Fourier
    coefficients satisfy
    \begin{equation*}
        \sum_{m,n \in \Z}|\hat{f}(m,n)|^q=\infty.
    \end{equation*}
    \end{lemma}

    \begin{proof}
    The fact that   $\b<1+ {1}/{\a}$ implies $f\in L^1(Q)$
    follows from \eqref{e:growth of f_a,b}, as can be
    easily verified using inequality \eqref{secret trick}.

    For any $n \geq 2$ and $m=2k$, with $k \in \N$, we estimate $\abs{\hat{f}(m,n)}$ from below.
    We write $h(x)=1-|x-{1}/{2}|^{\a}$,
    and note that for $x\neq {1}/{2}$ we have $0<h(x)<1$. 
    With this, we evaluate
    \begin{equation}
    \begin{split}    \label{e:first estimate for fourier coef]}
        \hat{f}(m,n)
        &=\int_0^1\int_0^1\frac{\e^{-2\pi \im (mx+ny)}}{(1+h(x)\e^{2\pi \im y})^{\b}}\dif x\dif y
        \\
        &=
        \int_0^1\e^{-2\pi i mx}\int_0^1\e^{-2\pi \im ny}\left(\sum_{j=0}^{\infty}(-1)^{j}b_{j} h^{j}(x)\e^{2\pi i j y}\right)\dif y\dif x
        \\
        &=
        (-1)^n b_n \underbrace{\int_0^1 h^n(x)\e^{-2\pi i mx}\dif x}_{(I)},
    \end{split}
    \end{equation}
	where 
    $b_n={\b (\b+1) \cdots (\b+n-1)}/{n!}$ are the coefficients of the Taylor expansion of $(1 - z)^{-\beta}$ at the origin.    %
   It follows by the product formula for the Gamma function that  we have $c n^{\b -1 } \leq b_n \leq C n^{\b-1}$, where $c$ and $C$ are positive constants.

     By a change of variables and the fact that $(1-|x|^{\a})^n$ is even, we get
    \begin{align*}
 		(I)
        &=
        2\int_0^{\frac{1}{2}}\left(1-x^{\a}\right)^n\cos 2\pi mx\dif x
        =
        \frac{1}{k}\int_0^k\bigg(1-\left(\frac{x}{2k}\bigg)^{\a}\right)^n\cos 2\pi x\dif x.
    \end{align*}
    We integrate by parts and find that the last expression is equal to
    \begin{equation}
    \label{e:second estimate for fourier coef}
        C \frac{n}{k^{\a+1}} \int_0^k\left(1-\left(\frac{x}{2k}\right)^{\a}\right)^{n-1}{x^{\alpha-1}}\sin 2\pi x\dif x.
    \end{equation}
    The function $(1-|{x}/{2k}|^{\a})^{n-1}x^{\a-1}$
    is decreasing on $(0,k)$ since $0<\a<1$. So for any positive integer $\nu < k$ the integral
    $\int_\nu^{\nu+1}(1-|{x}/{2k}|^{\a})^{n-1}x^{\alpha-1}\sin 2\pi
    x\dif x$ is positive. Using this and \eqref{e:second estimate for fourier coef}, we get
     \begin{equation}
    	\label{e:third estimate for fourier coef}
       (I) \geq C \frac{n}{k^{\a+1}} \underbrace{\int_0^1\left(1-\left(\frac{x}{2k}\right)^{\a}\right)^{n-1}{x^{\a-1}}\sin 2\pi x\dif x}_{(II)}.
    \end{equation}
    We use the same type of argument again to find that 
    \begin{equation}     
    	\begin{split} \label{e:fourth estimate for fourier coef}
		 (II) &\geq 
        \left(\int_0^{\frac{1}{4}}+\int_{\frac{3}{4}}^1\right)\left(1-
        \left(\frac{x}{2k}\right)^{\a}\right)^{n-1}x^{\a-1}\sin 2\pi x\dif x
        \\
        &\geq
		C\left[\left(1-\left(\frac{1}{8k}\right)^{\a}\right)^{n-1}-\left(1-\left(\frac{3}{8k}\right)^{\a}\right)^{n-1}\right].
	\end{split}
    \end{equation}
    Set $F(x)=C(1-({x}/{k})^{\a})^{n-1}$. 
    There exists a number ${1}/{8}<\tau<{3}/{8}$ such
    that the right-hand side of \eqref{e:fourth estimate for fourier coef} is equal to $4^{-1} F'(\tau)$. 
    It now follows that
    \begin{equation*} \label{ihaw}
		(II)        \geq 
        C \frac{n}{k^{\a}}\e^{n\log \left(1-\left(\frac{3}{8k}\right)^{\a}\right)}.
    \end{equation*}
    For  $0<x\leq{3}/{8}$, we have $\log (1-x)\geq -2x$  (actually,
    even for bigger $x$). Combining this with \eqref{e:first estimate for fourier
    coef]}, \eqref{e:third estimate for fourier coef}, and the asymptotic behaviour of $b_n$, we find that for constants $C_1>0$ and
    $C_2>0$ depending only on $\a$ and $\b$, we have
    \[
    |\hat{f}(m,n)|\geq
    C_1\frac{n^{\b+1}}{k^{2\a+1}}\e^{-C_2\frac{n}{k^{\a}}}
    \]
    whenever $n\geq2$ is an integer and $m=2k>0$ is an even integer.

    We are now ready to estimate the $\ell^q$ norm of the Fourier
    coefficients of $f$. First, 
    \begin{align*}
        \sum_{m,n \in \Z}|\hat{f}(m,n)|^q
        \geq
        C_1 \sum_{k=1}^{\infty} \frac{1}{k^{q(2\alpha+1)}} \underbrace{\sum_{n=2}^{\infty} {n^{q(\b+1)}}\e^{-C_2 q\frac{n}{k^{\a}}}}_{(*)}.
	\end{align*}
	 For a positive number $x$, we denote by $\lceil x\rceil$ the smallest integer $l$ such that    $l \geq x$. In this way,
    \begin{align*}
        (*) &\geq  \sum_{\nu=1}^{\infty}\sum_{n=\nu\lceil k^{\a}\rceil}^{(\nu+1)\lceil k^{\a}\rceil-1}
          n^{q(\b+1)} \e^{-C_2q\frac{n}{k^{\a}}} 
        \\
        &\geq
         k^{\a} \sum_{\nu=1}^{\infty} 
         {(\nu k^{\a})^{q(\b+1)}} \e^{-2C_2 q\frac{(\nu+1)k^{\a}}{k^{\a}}} 
        \\
        &=
         k^{\a(1 + q\beta + q)} \underbrace{\sum_{\nu=1}^{\infty} \nu^{q(\b+1)}\e^{-2C_2q(\nu+1)}}_{(**)}.
    \end{align*}
	Since $(**)$ converges,  
    \begin{equation*}
        \sum_{m,n \in \Z}|\hat{f}(n,m)|^q \geq
	      C \sum_{k=1}^{\infty} k^{\a+q(\a \b-\a-1)}.
	\end{equation*}
   The right-hand side is infinite if and only if $\b \geq (1+1/{\a})(1-1/{q})$, which gives the desired conclusion.
    \end{proof}

\subsection{Proof of Theorem \ref{main theorem}, part (a)}

Let $r \leq s$ be such that the point $(1/r, 1/s)$ is below the curve $\Gamma_q$, given by \eqref{gamma q}. 
This implies that either one of the following conditions holds:
   \begin{equation}
   \label{case two for part a}
   \frac{1}{r} + \frac{3}{s} > 1 \quad \text{and} \quad \frac{1}{r} + \frac{1}{s} < \frac{q}{2(q-1)},
   \end{equation}
   or
 \begin{equation}
\label{case one for part a}
   \frac{1}{r} + \frac{3}{s} \leq 1 \quad \text{and} \quad \frac{3q-2}{q+2}\cdot \frac{1}{r} + \frac{1}{s} < 1.
   \end{equation}  
    Moreover, since $G(g,1,1)$ is exact, condition
    \eqref{e:bl cond for exact, c-d} implies that in both of these cases we also have
   \begin{equation}
   \label{e:3r+1s>11}
    \frac{3}{r}+\frac{1}{s}>1.
    \end{equation}

    To arrive at a contradiction, we assume that the integrals in \eqref{e:bl cond r-s} converge. 
   Since both conditions  \eqref{case two for part a} and \eqref{case one for part a} imply that 
   the numbers $r$ and $s$ satisfy the inequality $1/r+1/s<1$, it follows 
   by Lemma \ref{prelim: zero of zac} that there exists a point $(a,b)\in Q$ such that $Zg(a,b) = 0$. 
    Therefore Lemma \ref{zak lemma} implies the estimate
    \begin{equation}\label{e:estimate for zacs size :-)}
        \abs{Zg(x,y) }^2 \leq C\Big( \phi_{r,s}(x-a) + \phi_{s,r}(y-b) \Big), \qquad (x,y) \in Q,
    \end{equation}
    where the functions $\phi_{r,s}$ are defined in \eqref{e:who is phi}. Note that the value of $\phi_{r,s}(x-a)$ is determined 
    by the inequality \eqref{e:3r+1s>11}, while the value of $\phi_{s,r}(y-b)$ is determined by the  left-hand inequality in 
    either \eqref{case two for part a} or \eqref{case one for part a}, depending on the case.

    By Lemma \ref{l:basic lemma, function version}$(b)$,  a contradiction is obtained if we find a function $h$ that
   satisfies
        \begin{equation} \label{point of main proof 1}
          h \in L^2_{\abs{Zg}}(Q)
        \end{equation}
        and
        \begin{equation} \label{point of main proof 2}
        h \in L^1(Q) \quad \text{with}  \quad \sum_{m,n \in \Z} \abs{\hat{h}(m,n)}^q = \infty.
        \end{equation}
     Roughly speaking, we construct a function $h$ that has a single singularity at the point $(a,b)$, and on the one 
    side, grows fast enough near this singularity for condition \eqref{point of main proof 2} to hold, 
    while on the other side, it grows slowly enough for condition \eqref{point of main proof 1} to follow from
     \eqref{e:estimate for zacs size :-)}. In fact, its size is essentially smaller than some power of $1/|Zg|^2$.

    Given $\alpha, \beta >0$ let $\tilde{f}_{\alpha,\beta}$ be the $1$-periodic extension of the function \eqref{e: def of f, a,b} to $\R^2$. Set 
    \begin{equation}
        \label{e:def of h_a,b}
        h_{\a,\b}(x,y) = \tilde{f}_{\alpha,\beta}\Big(x-a + \frac{1}{2}, y - b + \frac{1}{2} \Big).
    \end{equation}
    From \eqref{e:growth of f_a,b}, we have
    \begin{equation}
        \label{e:growth of h_a,b}
        |h_{\a,\b}(x,y)|^2\leq
        \frac{C}{\left(|x-a|^{2\a}+|y-b|^2\right)^{\b}}, \qquad \forall (x,y) \in Q.
    \end{equation}
    In the remainder of the proof, 
    we determine suitable values for the parameters $\alpha,\beta$ to ensure that 
    \eqref{point of main proof 1} and \eqref{point of main proof 2} hold for $h_{\a,\b}$.

    First, we assume that \eqref{case two for part a}
    holds. The condition
    ${1}/{r}+{1}/{s}<{q}/{(2(q-1))}$ implies that there
    exists a number $\l$ which satisfies
    \begin{equation}
    \label{e:cond on l}
    \frac{2(q-1)}{q}\cdot
    \left(\frac{1}{r}+\frac{1}{s}\right)<\l<1.
    \end{equation}
Choose such a $\l$ and define $h_{\alpha,\beta}$ as in \eqref{e:def of h_a,b} with
\begin{equation}
\label{e:who are a and b first case}
\a=\frac{r}{s} \qquad \textrm{and} \qquad \b=\frac{s\l}{2}.
\end{equation}
    Since $r\leq s$, the left inequality in \eqref{case two for part a} implies that 
        ${1}/{r}+{1}/{s}>{1}/{2}$. Combining this with \eqref{e:cond on l} and \eqref{e:who are a and b first case}, we have
    \begin{equation*}
        \left( 1 - \frac{1}{q} \right) \left( 1 + \frac{1}{\alpha} \right) \leq \beta < \left( 1 + \frac{1}{\alpha} \right).
    \end{equation*}
    So,  by Lemma \ref{coefficient lemma} the function $h_{\alpha,\beta}$
    satisfies \eqref{point of main proof 2}.
To  show that it satisfies \eqref{point of main proof 1}, we first note that, in this case, the inequality \eqref{e:estimate for zacs size :-)} takes the form
    \begin{equation*}
        \abs{Zg(x,y)}^2 \leq \abs{x-a}^{2 + r(1 - \frac{3}{r} - \frac{1}{s} )}
        + \abs{y-b}^{ 2 + s(1 - \frac{1}{r} - \frac{3}{s}) }.
    \end{equation*}
    Combining this with \eqref{e:growth of h_a,b} and \eqref{e:who are a and b first case}, we use \eqref{secret trick} to get
    \begin{multline*}
        \iint_Q  \abs{{h}_{\alpha,\beta}(x,y)}^2 \abs{Zg(x,y)}^2 \dif x \dif y
        \leq
        C \iint_Q \left( \abs{x-a}^{\frac{r}{s}}
        + \abs{y-b}\right)^{\left[2 + s(1 - \frac{1}{r} -
        \frac{3}{s})-s\l \right]}\dif x \dif y.
    \end{multline*}
    The integral on the right-hand side is finite if and only if
    $\l<1$, so \eqref{point of main proof 1} follows from \eqref{e:cond on
    l}.

    We now assume that \eqref{case one for part a} holds.
    Let
    \begin{equation}
    \label{e:who is a, second case}
  \alpha = 1 - \frac{r}{2} \left(  \frac{3}{r} + \frac{1}{s}-1 \right),
    \end{equation}
    and note that $0<\alpha < 1$.
    Next, the condition $\frac{3q-2}{q+2}\cdot \frac{1}{r}+\frac{1}{s} < 1$
    implies that there exists a number $\b$ for which
    \begin{equation}
    \label{e:who is b secon case}
\left(1 - \frac{1}{q} \right) \left( 1 + \frac{1}{\alpha} \right)
\leq \beta <\frac{3}{2} +  \frac{1}{2\alpha}.
    \end{equation}
Choose such a $\b$ and  let $h_{\a,\b}$ be the function
defined in \eqref{e:def of h_a,b}. Since $0<\alpha<1$,
\begin{equation*}
     \frac{3}{2} + \frac{1}{2\alpha} < 1 + \frac{1}{\alpha},
\end{equation*}
and so it follows from \eqref{e:who is b secon case} and Lemma \ref{coefficient lemma} that
\eqref{point of main proof 2} holds.
To check that \eqref{point of main proof 1} holds, 
we first note that, in this case, inequality \eqref{e:estimate for zacs size :-)} takes the form
    \begin{align*}
        \abs{Zg(x,y)}^2
        &\leq
       C\left( \abs{x-a}^{2\a} + \abs{y-b}^2 \log \left( 1+ \frac{1}{|y-b|}\right) \right)
        \\
        &\leq
        C \left(\abs{x-a}^{2\a} + \abs{y-b}^2\right) \log\left(1+ \frac{1}{|y-b|}\right).
    \end{align*}
Combining  this with the estimate in \eqref{e:growth of h_a,b}, and making an appropriate change of variables, we find that
\begin{equation*}
        \iint_Q \abs{h_{\a,\b}(x,y)}^2 \abs{Zg(x,y)}^2 \dif x \dif y
        \leq
        C \iint_Q \left(x^{2\a}+ y^2\right)^{1-\b} \log\left(1+ \frac{1}{|y|}\right)\dif x \dif y.
\end{equation*}
    We use \eqref{secret trick} to ensure that the last integral is
    smaller than
    \begin{equation*}
        C \iint_Q \left(x
        + y^{\frac{1}{\a}}\right)^{2\a(1-\b)} \log\left(1+ \frac{1}{|y|}\right)\dif x \dif y.
    \end{equation*}
This integral is finite if $\beta
<{3}/{2} +  {1}/{2\alpha}$. Hence,
\eqref{point of main proof 1}
follows from the inequality \eqref{e:who is b secon case}.

    For $s=\infty$, use Remark \ref{phi remark} and repeat the previous argument.
\qed

 \section{Two families of functions} \label{building blocks section}
We introduce two families of functions that are used in the next Section  to prove part $(b)$ of Theorem \ref{main result}. 
 The needed estimates are given in lemmas \ref{part2_modulus lemma} and \ref{part2_full lemma}, where we measure the smoothness of these functions near the origin.


\subsection{Building blocks for the modulus}
   Fix $a>0$. Given $\alpha,\beta,\gamma >0$, set
    \begin{equation} \label{part2_definition of f}
        f_{\alpha,\beta,\gamma}(x,y) := \left\{ \begin{array}{cc} \big(x^{{\alpha}/{\gamma} } +
        \abs{y}^{{\beta}/{\gamma}}\big)^\gamma & \quad \text{for} \; x \geq 0, \\ \big((-ax)^{{\alpha}/{\gamma} } +
        \abs{y}^{{\beta}/{\gamma}}\big)^\gamma & \quad \text{for} \; x < 0. \end{array} \right.
    \end{equation}
   The following lemma is easily proved by induction.
    \begin{lemma} \label{part2_basic lemma}
    Let $\alpha,\beta>0$ and $k\in \N$.
    If $0<\gamma < \min\{ \alpha/k,\beta/k, 1\}$, then $f_{\alpha,\beta,\gamma} \in C^k(\R^2 \backslash \{ (0,0)\})$. Moreover, for any
     $(x,y) \neq (0,0)$, the partial derivative
     $(f_{\alpha,\beta,\gamma})_x^{(k)}(x,y)$ equals
        \begin{equation}
        \label{part2_e;derivative of f_abc}
                        \left \{ \begin{array}{cc} \sum_{m=1}^k C_{m,k} (x^{\alpha/\gamma} + \abs{y}^{\beta/\gamma})^
            {\gamma-m} x^{m \frac{\alpha}{\gamma} - k} & \quad x \geq 0, \\ (-a)^k \sum_{m=1}^k C_{m,k}
            \Big( (-ax)^{\alpha/\gamma} + \abs{y}^{\beta/\gamma} \Big)^{\gamma-m}  (-ax)^{m \frac{\alpha}{\gamma} - k}
            & \quad x < 0, \end{array}  \right.
        \end{equation}
        where $ C_{m,k}$ are constants not depending on $(x,y)$.
    \end{lemma}

An explicit estimate for the smoothness of the functions
$f_{\alpha,\beta,\gamma}$ near the origin is given in the
following lemma. Recall that $\Delta_h$ and $\Gamma_h$ are defined in Section \ref{smoothness and zak subsection}.

\begin{lemma} \label{part2_modulus lemma}
   Let $\alpha, \beta> 0$ and $k \in \N$ be such that $2\alpha + \alpha/\beta + 1 \leq 2k$.
   If $\gamma <  \alpha/k$, then for any $\epsilon>0$ we
   have
    \begin{equation*}
        \int_{-b}^b \iint_{[-c,c]^2} \frac{\abs{\Delta_h^k f_{\alpha,\beta,\gamma}(x,y)}^2}{\abs{h}^
        {2\alpha + \alpha/\beta + 2  - \epsilon}} \dif x \dif y \dif h  < \infty,
    \end{equation*}
    where $b,c$ are any two positive numbers.
\end{lemma}
\begin{proof}
    Set $f = f_{\alpha,\beta,\gamma}$. To simplify formulations, we make the assumption $\gamma <  \beta/k$
    so that Lemma \ref{part2_basic lemma} can be applied.
    Otherwise, a relaxed version of it, where the
    function $f$ is not necessarily differentiable, can be used. However, in what follows, this extraneous
    condition holds whenever we refer to Lemma \ref{part2_modulus
    lemma}.

    In the  above integral the integrand is even in $y$, so it is enough to show that for $ h > 0$,  
    \begin{equation}\label{part2_e: what we need to show in lemma3}
     \int_0^c \int_{-c}^c \Big( \abs
    {\Delta_h^k f(x,y)}^2 + \abs{\Delta_{-h}^k f(x,y) }
        ^2 \Big) \dif x \dif y \leq C h^{2\alpha + \alpha/\beta + 1},
    \end{equation}
    where $C=C(f,k)$ does not depend on $h$.
   Since $f$ is bounded on any compact set, we may assume that $h$ is small enough for the following partition\footnote{To simplify formulations, here and in the following, we allow members of a partition to have intersections of measure zero.}
   to hold:
    \begin{equation*}
        [-c,c] \times [0,c] = V_1 \cup V_2 \cup V_3 \cup V_4,
    \end{equation*}
    where (see Figure \ref{part2_V figure1})
    {\begin{align*}
        V_1 & =  [-(k+1)h, (k+1)h] \times [0, ((k+1)h)^{\alpha/\beta}], \\
        V_2 & =  [-(k+1)h, (k+1)h] \times [((k+1)h)^{\alpha/\beta}, c], \\
        V_3 & =  [(k+1)h,c] \times [0,c], \\
        V_4 & =  [-c, -(k+1)h] \times [0,c].
    \end{align*}}

    \begin{figure}
        \includegraphics{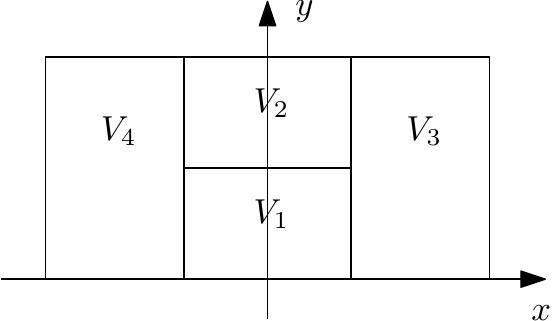}
        \caption{The partition of $[-c,c] \times [0,c]$   in the proof of Lemma \ref{part2_modulus lemma}.}
       \label{part2_V figure1}
    \end{figure}

   To estimate the integral in (\ref{part2_e: what we need to show in lemma3}) over $V_1$,
   we use Remark \ref{part2_delta and gamma remark}$(a)$ and the inequality \eqref{secret trick} to find
   that it is smaller than some constant times 
      \begin{align*}
         \sum_{j=-k}^k \iint_{V_1} \abs{f(x+jh,y)}^2 \dif x \dif y
        &\leq
        C \int_0^{((k+1)h)^{\alpha/\beta}} \int_{-(2k+1)h}^{(2k+1)h} \abs{f(x,y)}^2 \dif x \dif
        y  \\
       & \leq C\int_0^{((k+1)h)^{\alpha/\beta}} \int_0^{(2k+1)h} \Big( x^{2\alpha} + y^{2\beta} \Big) \dif x \dif
       y\\
        & \leq C h^{2\alpha + \alpha/\beta + 1}.
\end{align*}
    For the estimate over the remaining parts, we note that $\gamma<\min\{\alpha/k,\beta/k\} $ implies $\gamma < 1$, and
    so it follows from 
     Lemma \ref{part2_basic lemma} and Remark \ref{part2_delta and gamma remark}$(c)$
	that
    \begin{equation*}
        \abs{\Delta_h^k f(x,y)} + \abs{\Delta_{-h}^k f(x,y)} \leq  \abs{h}^{k}
        \underbrace{\sup_{\xi \in [x-kh,x+kh]} \abs{f_x^{(k)}(\xi,y)}}_{:=
        \Omega(x,y)}.
    \end{equation*}
   Hence, to complete the proof we need to show that
    \begin{equation*}
    I_2+I_3+I_4:=    \iint_{V_2 \cup V_3 \cup V_4} \Omega(x,y)^2
    \dif x \dif y \leq C h^{2\alpha + \alpha/\beta + 1 - 2k}.
    \end{equation*}
We do so by estimating the partial derivatives given in
(\ref{part2_e;derivative of f_abc}). For $(x,y) \in V_2$ we have
    \begin{equation*}
         \Omega(x,y) \leq C \sum_{m=1}^k y^{\frac{\beta}{\gamma} (\gamma-m) } h^{m \frac{\alpha}{\gamma} - k }\leq
         C y^{\beta-k\beta/\alpha}.
    \end{equation*}
  So,
    \begin{equation*}
        I_2 \leq Ch(1+h^{2\alpha + \alpha/\beta - 2k})\leq  C h^{2\alpha +
\alpha/\beta + 1 - 2k},
    \end{equation*}
    whenever $h$ is small enough.

To estimate $I_3$, we first note that if $(x,y) \in V_3$, then
    \begin{equation*}
          \Omega(x,y) \leq
        \sup_{\xi \in [x/(k+1), 2x]}   \abs{f_x^{(k)}(\xi,y)} \leq
        C \sum_{m=1}^k (x^{\alpha/\gamma} + y^{\beta/\gamma})^{\gamma-m} x^{m \frac{\alpha}{\gamma} - k}.
    \end{equation*}
 We apply \eqref{secret trick} to the $m$'th
term of this sum
    and find that the corresponding integral is less than some constant times
    \begin{align*}
        \int_{(k+1)h}^c \int_0^c (x^{\alpha/\beta} + y)^{2 \frac{\beta}{\gamma} (\gamma-m) } x^{2 (m \frac{\alpha}{\gamma} - k)} \dif y \dif x
        &\leq
        C\left( 1  + \int_{(k+1)h}^c x^{2\alpha + \alpha/\beta - 2k} \dif x \right) \\
        &\leq  C h^{2\alpha + \alpha/\beta + 1 - 2k},
    \end{align*}
   whenever $h$ is small enough. This implies the required
   estimate for $I_3$.
   In the same way one can show the required estimate for $I_4$, which completes the proof.
\end{proof}


\subsection{Building blocks for the argument} \label{part2_kokkeliko}
Let $\phi\in C^\infty(\R)$ be a function satisfying $-1 \leq
\phi(x) \leq 0$ for all $x \in \R$, and for which
\begin{equation*}
     \phi(x) = \left\{ \begin{array}{cc} -1 & \quad x \in (-\infty,0], \\
     0 & \quad [1,\infty). \end{array} \right.
\end{equation*}
Given $\lambda>0$, denote
\begin{equation} \label{part2_definition of H}
     H_{\lambda}(x,y) = \left\{ \begin{array}{cl}
     \phi\big(\frac{y}{x^{\lambda}} \big) & \quad \text{for} \quad  x \geq 0, \quad \text{and} \quad  0 \leq y \leq x^{\lambda}, \\
    0 & \quad  \mathrm{otherwise}.
     \end{array}  \right.
\end{equation}
Such a function was first introduced in \cite{benedetto_czaja_gadzinski_powell2003}.
   For $\alpha,\beta,\gamma>0$, set
\begin{equation} \label{part2_definition of F}
     F_{\alpha,\beta,\gamma}(x,y) = f_{\alpha,\beta,\gamma}(x,y) \e^{2\pi \im
     H_{\frac{\alpha}{\beta}}(x,y)},
\end{equation}
where the functions $f_{\alpha,\beta,\gamma}$ are defined in \eqref{part2_definition of f}.
The following lemma, combined with Lemma \ref{part2_basic lemma},
provides a preliminary estimate for the smoothness of the
functions $F_{\alpha,\beta,\gamma}$. These estimates are easily
obtained by an   inductive process.

\begin{lemma} \label{part2_derivatives of H}
   Let $\lambda>0$. The function $\e^{2\pi \im H_\lambda(x,y)}$ belongs to $C^{\infty}(\R^2\setminus \{0,0\})$. Moreover,
   for any $(x,y)\neq (0,0)$ and $n \in \N$,
    we have.
    \begin{align*}
         (a) && \abs{(\e^{2\pi \im H_\lambda})_x^{(n)}(x,y)} &\leq \left\{  \begin{array}{cc} C y
         x^{-\lambda -n} & \quad 0<x , \; 0<y<x^{\lambda}, \\
         0 & \mathrm{otherwise}, \end{array}  \right. \\
         (b) && \abs{(\e^{2\pi \im H_\lambda})_y^{(n)}(x,y)} &\leq \left\{  \begin{array}{cc}
         C x^{-n\lambda } & \quad 0<x, \; 0<y<x^{\lambda}, \\
         0 & \mathrm{otherwise}, \end{array}  \right.
    \end{align*}
     where $C=C(n,\lambda)$ does not depend on $x$ and $y$.
\end{lemma}

An explicit estimate for the smoothness
of the functions $F_{\alpha,\beta,\gamma}$ near the origin is given in the following lemma.
\begin{lemma} \label{part2_full lemma}
    Let $\alpha, \beta> 0$ and $k \in \N$ be such that both $2\alpha + \alpha/\beta + 1 \leq 2k$ and
    $2\beta + \beta/\alpha + 1 \leq 2k$. If $\gamma < \min \{ \alpha/k, \beta/k \}$, then for any $\epsilon>0$  we have
    \begin{align*}
         (a)&& \int_{-b}^b  \iint_{[-c,c]^2} \frac{\abs{\Delta_h^k F_{\alpha,\beta,\gamma}(x,y)}^2}
         {\abs{h}^{2\alpha + \alpha/\beta + 2 - \epsilon}} \dif x \dif y \dif h < \infty,\\
         (b)&& \int_{-b}^b   \iint_{[-c,c]^2} \frac{\abs{\Gamma_h^k F_{\alpha,\beta,\gamma}
         (x,y)}^2}{\abs{h}^{2\beta + \beta/\alpha + 2 - \epsilon}} \dif x \dif y \dif h < \infty,
    \end{align*}
    where $b,c$ are any two positive numbers.
\end{lemma}

\begin{proof}
 We show that the integral in
$(a)$ converges. For the estimate of the integral in $(b)$, which
can be obtained in much the same way, we give a short sketch at
the end of this proof.

    Set $f=f_{\alpha,\beta,\gamma}$, $F=F_{\alpha,\beta,\gamma}$ and $H=H_{\frac{\alpha}{\beta}}$.
     By Remark \ref{part2_delta and gamma remark}$(b)$, it
      is enough to show that for every $0 \leq n \leq k$ and $0<h\leq1$
      we have
    \begin{equation}
        \label{part2_first argument inequality}
        \iint_{[-c,c]^2} \Big( \abs{\Delta_h^n  \e^{2\pi \im H } \cdot \Delta_h^{k-n} f_{nh} }^2 +  \abs{\Delta_{-h}^n
         \e^{2\pi \im H } \cdot \Delta_{-h}^{k-n}f_{-nh}}^2 \Big) \dif x \dif y \leq C h^{2\alpha + \alpha/\beta + 1},
    \end{equation}
    where we use the notation $f_{nh}(x,y)=f(x+nh,y)$ and the constant $C=C(F,k)$ does not depend on $h$.
    Since the case $n=0$ follows from Lemma \ref{part2_modulus lemma}, it remains to show that (\ref{part2_first argument inequality})
 holds for $1 \leq n \leq
    k$. Fix such an integer $n$.

     \begin{figure}
        \includegraphics{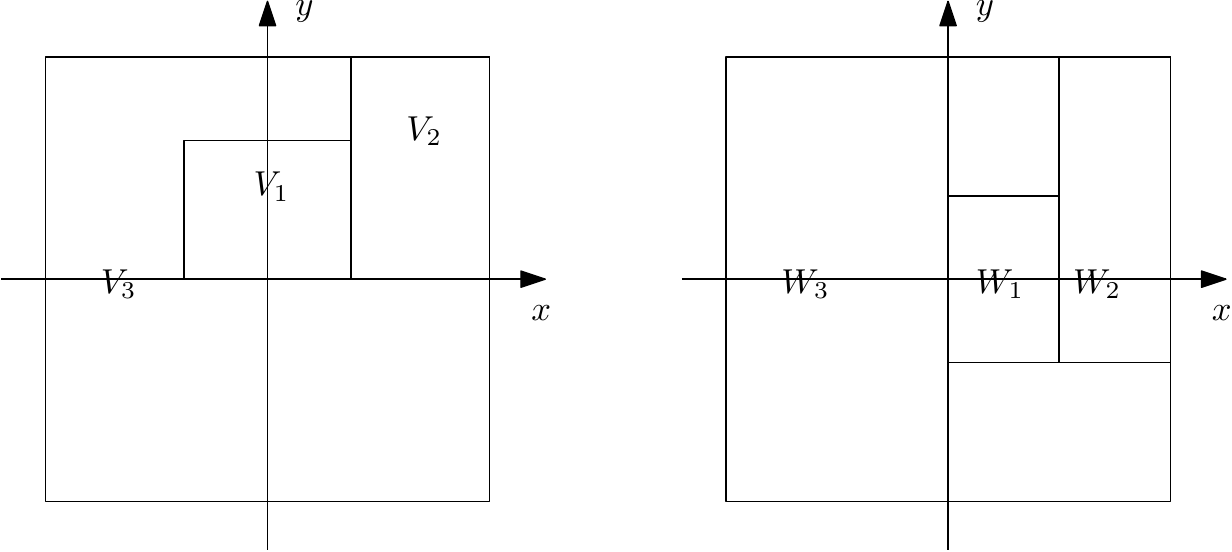}
        \caption{Partition of $[-c,c]^2$ into $V_1 \cup V_2 \cup V_3$ and $W_1 \cup W_2 \cup W_3$ in the proof of Lemma \ref{part2_full lemma}.}
        \label{part2_V figure2}
    \end{figure}

    Since $F$ is bounded on any compact set, we can assume that $h$ is small enough for the following partition to hold
    \begin{equation*}
         [-c,c]^2 = V_1 \cup V_2 \cup V_3,
    \end{equation*}
    where (see Figure \ref{part2_V figure2})
    \begin{align*}
        V_1 & =  [-(k+1)h, (k+1)h] \times [0, ((2k+1)h)^{\alpha/\beta}], \\
        V_2 & =  [(k+1)h,c] \times [0,c],\\
        V_3 & =  [-c,c]^2 \backslash (V_1 \cup V_2). 
    \end{align*}
    This induces the splitting $I_1 + I_2 + I_3$ of the integral
    in (\ref{part2_first argument inequality}).
 The estimate for $I_1$ can be obtained using Remark \ref{part2_delta and gamma remark}$(a)$, as was done for the estimate over the
     area  $V_1$ in the proof of Lemma \ref{part2_modulus lemma}.
    To estimate $I_3$, it suffices to observe that $\e^{2\pi \im H(\xi,y)} = 1$ for $(x,y) \in V_3$
    and $\xi \in [x-kh, x+kh]$, whence $I_3 = 0$ (for
    example, by Remark \ref{part2_delta and gamma remark}$(c)$).

    We estimate $I_2$. By lemmas \ref{part2_basic lemma} and \ref{part2_derivatives of H}, one can apply Remark
    \ref{part2_delta and gamma remark}$(c)$ in this area. So, it suffices to show that
    \begin{equation*}
        \iint_{V_2}  \Big( \sup_{\abs{\xi-x} \leq kh} \abs{ (\e^{2\pi \im H(\xi,y)})_x^{(n)} }^2 \Big)
        \Big(  \sup_{\abs{\xi-x}\leq kh} \abs{ f_x^{(k-n)}(\xi,y) }^2 \Big) \dif x \dif y
        \leq h^{2\alpha + \alpha/\beta + 1 - 2k}.
    \end{equation*}
    Note that for $(x,y) \in V_2$, the condition $\abs{\xi-x} \leq kh$ implies that
    $x/(k+1) \leq \xi \leq 2x$.
    In addition, if $y > (2x)^{\alpha/\beta}$, then $(\e^{2\pi \im H(\xi,y)})_x^{(n)} = 0$.
    On the other hand, if $y \leq (2x)^{\alpha/\beta}$, then by Lemmas \ref{part2_basic lemma} and \ref{part2_derivatives of H} we have
    \begin{multline*}
        \sup_{\abs{\xi-x} \leq kh}  \abs{ (\e^{2\pi \im H(\xi,y)})_x^{(n)} }
        \sup_{\abs{\xi-x} \leq kh}  \abs{ f_x^{(k-n)}(\xi,y) }
        \\
        \leq
        \sup_{x/(k+1) \leq \xi \leq 2x}  \abs{ (\e^{2\pi \im H(\xi,y)})_x^{(n)} }
        \sup_{x/(k+1) \leq \xi \leq 2x}  \abs{ f_x^{(k-n)}(\xi,y) }
        \leq
        C y x^{\alpha - \alpha/\beta -k}.
    \end{multline*}
    Hence, the desired estimate is found by checking that
    \begin{equation*}
        \int_{(k+1)h}^c \int_0^{(2x)^{\alpha/\beta}}  y^2 x^{2\alpha - 2\alpha/\beta -2k}  \dif y \dif x
        \leq
        C h^{2\alpha + \alpha/\beta + 1 - 2k}.
    \end{equation*}
    This completes the proof of \eqref{part2_first argument inequality}.

     In a similar way, one can show that the integral in $(b)$
    converges. An appropriate partition in this case is
 $[-c,c]^2 = W_1 \cup W_2 \cup W_3,$
    where (see Figure \ref{part2_V figure2})
    \begin{align*}
        W_1  &=  [0,h^{\beta/\alpha}] \times [-(k+1)h, (k+1)h], \\
        W_2  &=  [h^{\beta/\alpha},c] \times [-(k+1)h, c], \\
        W_3  &=  [-c,c]^2 \backslash (W_1 \cup W_2).
    \end{align*}

    Note that over $W_2$, the corresponding integral is smaller than
     the one taken over the area $h^{\beta/\alpha} < x < c$ and $- (k+1)x^{\alpha/\beta} < y <
     (k+1)x^{\alpha/\beta}$, which is easily estimated.
    This completes the proof.
\end{proof}

%
%

%

%

%


\section{Theorem \ref{main theorem} -- Second part} \label{proof - part b} \label{final section}

Here we prove part $(b)$ of Theorem \ref{main result}. Assume that the point $(r,s)$ is above the curve $\Gamma_q$ (see \eqref{gamma q} and Figure \ref{figure}). This implies that either one of the following conditions holds: 
   \begin{equation}
   \label{case two for part b}
   \frac{1}{r} + \frac{3}{s} > 1 \quad \text{and} \quad \frac{1}{r} + \frac{1}{s} > \frac{q}{2(q-1)},
   \end{equation}
   or
   \begin{equation}
\label{case one for part b}
   \frac{1}{r} + \frac{3}{s} \leq 1 \quad \text{and} \quad \frac{3q-2}{q+2}\cdot \frac{1}{r} + \frac{1}{s} > 1.
   \end{equation}

    For each of the conditions 
    \eqref{case two for part b} and \eqref{case one for part b}, we a construct a quasi-periodic
    function $G$ on $\R^2$ that, when restricted to $Q$, is square integrable. By the surjectivity of the Zak transform, there exists a
    function $g\in L^2(\R)$ such that $g = Z^{-1}G$. We prove that this function satisfies
    all the requirements of Theorem \ref{main result}. Roughly
    speaking, we show that on the one hand the functions $G$ are smooth enough for the
    time--frequency conditions \eqref{e:bl cond r-s} to
    follow from Lemma \ref{part2_stein lemma for delta and gamma},  while
    on the other hand they decrease slowly enough, near their single
    zero, for the $(C_q)$ property to follow from Lemma \ref{l:basic lemma, function version}.

    In various stages of our construction we make simple
    interpolations of functions. To this end, we
    use of the following auxiliary function. Fix $0<\eta <1/4$, and denote by $\rho(t)$ an even
    function in $C^\infty(\R)$ which satisfies
    $0 < \rho(t) < 1$ on $(-2\eta,2\eta)$ and
    \begin{equation}
    \label{part2_e:the function rho}
        \rho(t) = \left\{ \begin{array}{cl} 1 & \quad \text{for} \quad t \in [-\eta,\eta], \\
        0 & \quad \text{for} \quad  t\in \R \backslash [-2\eta,2\eta]. \end{array} \right.
    \end{equation}

   $\:$

   \subsection{Proof for the first set of conditions} Fix $r \leq s $ for which condition
   \eqref{case two for part b} holds. We may assume that $1/r + 1/s \leq 1$. 
   Choose $\epsilon>0$ small enough for the numbers
    \begin{equation}
    \label{part2_e:who are r',s': first part}
        r' = r + \epsilon \quad \text{and} \quad s' = s + \epsilon
    \end{equation}
    to satisfy
    \begin{equation}
    \label{part2_cond on r' and s'}
        \frac{1}{r'} + \frac{1}{s'} > \frac{q}{2(q-1)}.
    \end{equation}
    Set
    \begin{equation} \label{part2_choice of alpha and beta 1}
        \alpha =  \frac{r'}{2} \Big( 1 - \frac{1}{r'} - \frac{1}{s'} \Big) \quad \text{and} \quad
        \beta =  \frac{s'}{2} \Big( 1 - \frac{1}{r'} - \frac{1}{s'} \Big).
    \end{equation}

    In the construction of the function $G$ described above, we consider the
   argument and   modulus separately. In fact, we construct
    functions $\Psi$ and $\Phi$ such that
    \begin{equation}\label{part2_G in terms of phi psi}
    G(x,y)=\Phi(x-1/2,y-1/2)e^{2\pi i \Psi(x-1/2,y-1/2)}.
    \end{equation}
    To  define the argument of $G$, i.e., the real valued function $2\pi \Psi$, we use a minor modification of a  construction from \cite{benedetto_czaja_gadzinski_powell2003}.  That is, instead of a singularity at the origin, we find it more convenient to use an argument with a singularity at $(1/2, 1/2)$. This also accounts for the translation of $1/2$ in the definition \eqref{part2_G in terms of phi psi}.

     We begin by defining the function $\Psi(x,y)$ on $[-1/2,1/2) \times [0, 1)$:
     \begin{equation*}
     \Psi(x,y) = \left\{ \begin{array}{cc} 0 & \quad x \in [-1/2,0], \\
     \rho (x)H_{\frac{\alpha}{\beta}}(x,y)+(1-\rho
    (x))(y-1/2) & \quad x \in [0,1/2). \end{array} \right.
\end{equation*}
 where $H_{\frac{\alpha}{\beta}}$ is the function defined in  \eqref{part2_definition of H} (see Figure \ref{part2_figure psi}).
   \begin{figure}
           \includegraphics{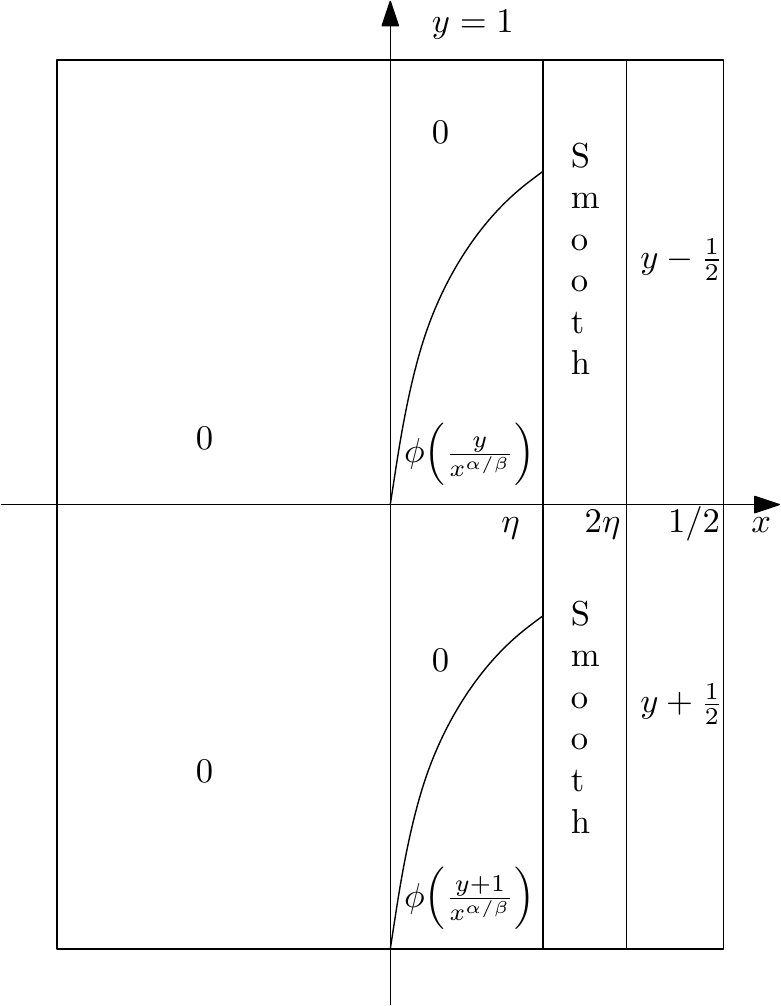}
       \caption{Illustration of the function $\Psi$ on $[-1/2,1/2] \times [-1,1]$.}
      \label{part2_figure psi}
   \end{figure}
	We extend $\Psi$ to the plane according to the rules
	\begin{equation} \label{rules in one line}
		\begin{split}
		\Psi(x+1,y)&=\Psi(x,y)+y-1/2 \qquad \text{for} \; x \in \R,\; y \in [0,1), \qquad \\
	 	\Psi(x,y+1)&=\Psi(x,y) \qquad \qquad \qquad \, \, \,  \text{for} \;  \;  (x,y) \in \R^2.
		\end{split}
	\end{equation}
   Note that the function $e^{2\pi i \Psi(x,y)}$ is continuous over $\R^2
    \setminus \Z^2$ since,
   \begin{equation*}
        \lim_{y \rightarrow  {1}^-} \Psi(x,y) = \left\{ \begin{array}{ll} \Psi(x,0) & \quad   \\
        \Psi(x,0) + 1 \end{array}\right. \qquad \text{for} \quad x\in [-1/2,1/2),
    \end{equation*}
    and
    \begin{equation*}
         \lim_{x \rightarrow \frac{1}{2}^-} \Psi(x,y) =  y -1/2 = \Psi(-1/2,y) + y - 1/2, \qquad \qquad \forall y \in [0,1).
    \end{equation*}
    In fact, one can verify that $\e^{2\pi \im \Psi(x,y)}$ belongs to $C^\infty(\R^2 \backslash \Z^2)$ (see also \cite{benedetto_czaja_gadzinski_powell2003}).

We turn to constructing the modulus of $G$, i.e. the function $\Phi$. Choose $k\in
\N$ which satisfies
\begin{equation}\label{part2_who is k, first part}
k>\frac{1}{2}\max\{r',s'\},
\end{equation}
and a number $\gamma>0$ for which
\begin{equation}\label{part2_who is gamma, first part}
\gamma<\min\left\{\frac{\alpha}{k},\frac{\beta}{k}\right\}.
\end{equation}
We define the function $\Phi(x,y)$ on $[-1/2,1/2)^2$ by
\begin{equation*}
        \Phi(x,y) = \rho(y) \Big( \rho(x) (\abs{x}^{\alpha/\gamma} + \abs{y}^{\beta/\gamma})^{\gamma} + 1 - \rho(x) \Big) + 1 - \rho(y),
    \end{equation*}
    and extend $\Phi$ to be a $1$-periodic function on the plane
    (See Figure \ref{part2_figure phi}).
   \begin{figure}
           \includegraphics{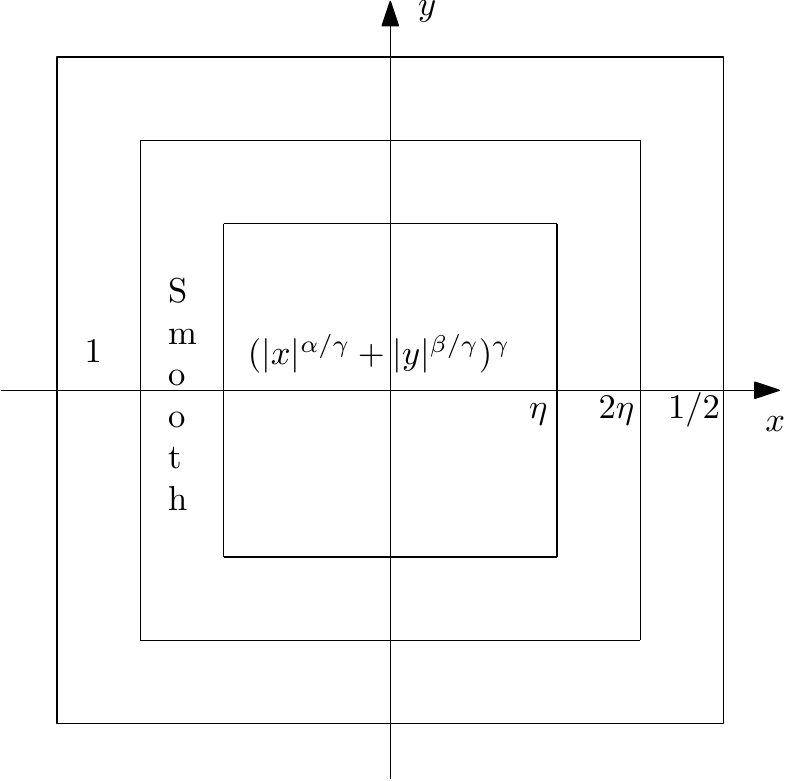}
       \caption{Illustration of the function $\Phi$ on $[-1/2,1/2]^2$.}
      \label{part2_figure phi}
   \end{figure}
    Note that  $\Phi$ is continuous on $\R^2$ since
    \begin{equation*}
        \Phi(-1/2,y)= \lim_{x \rightarrow \frac{1}{2}^-}\Phi(x,y)=1, \qquad \forall y\in [-1/2,1/2),
    \end{equation*}
    and
    \begin{equation*}
        \Phi(x,-1/2)= \lim_{y \rightarrow \frac{1}{2}^-}\Phi (x,y)=1, \qquad \forall x\in [-1/2,1/2).
    \end{equation*}
    In fact, using Lemma \ref{part2_basic lemma}, one can check
    that $\Phi \in C^k (\R^2 \backslash \Z^2)$. Moreover,
    $\Phi=0$ on the lattice $\Z^2$, and only there.

    Consider the function
    \[
    \Phi(x,y)e^{2\pi i \Psi(x,y)}.
    \]
    We list its growth and smoothness properties:
    \begin{itemize}
        \item[(i)] It belongs to $C^k (\R^2 \backslash \Z^2)$.
        \item[(ii)] Its modulus is continuous, bounded, and is equal to zero
        on $\Z^2$, and only there.
        \item[(iii)]There exists a neighbourhood of the origin, say $U$, on
        which it is equal to the function
        $F_{\alpha,\beta,\gamma}$ defined in \eqref{part2_definition of F}.
    \end{itemize}
    In particular, it follows that the function $G$ defined in \eqref{part2_G in terms of phi
    psi} is bounded. Moreover, since $\Phi$ is $1$-periodic, the condition \eqref{rules in one line} implies that $G$ is quasi-periodic over $\R^2$.
    Therefore, there exists a function $g\in L^2(\R)$ such that
    $Zg=G$ on $\R^2$. We prove that $g$ satisfies the requirements of
    Theorem \ref{main result}.

    To show that $g$ has the required
    time--frequency localisation, we check that the integrals in
    \eqref{e:bl cond r-s} are finite.
By Lemma \ref{part2_stein lemma for delta and gamma}, we need to show
that for $\alpha$, $\beta$, $k$ and $\gamma$ chosen above, the following integrals are
    finite:
    \begin{equation} \label{part2_what has to be finite}
            \int_{-\infty}^\infty \iint_{[0,1]^2} \frac{\abs{\Delta_h^k G(x,y)}^2}{h^{1+r}} \dif x \dif y \dif h
            \quad \text{and} \quad
            \int_{-\infty}^\infty \iint_{[0,1]^2} \frac{\abs{\Gamma_h^k G(x,y)}^2}{h^{1+s}} \dif x \dif y \dif h.
    \end{equation}
    We show that
    the left-hand integral is finite, the proof for the right-hand integral follows in the same way.
	As in Remark \ref{bonus remark}$(a)$, it is enough to show that for some $\delta >0$ the integral
     \begin{equation*}
        \int_{-\delta}^\delta \iint_{[0,1]^2} \frac{\abs{\Delta_h^k G(x,y)}^2}{h^{1+r}} \dif x \dif y \dif h=
        \int_{-\delta}^\delta \iint_{[-\frac{1}{2},\frac{1}{2}]^2}\frac{\abs{\Delta_h^k \big(\Phi (x,y)e^{2\pi
         i\Psi(x,y)}\big)}^2}{h^{1+r}} \dif x \dif y \dif h\\
         \end{equation*}
         converges.

         So, choose $\delta>0$ such that
         \begin{equation}
         \label{part2_who is delta, first part}
         [-(2k+1)\delta,(2k+1)\delta]^2\subset U,
         \end{equation}
         where $U$ is the neighbourhood of the origin described in
         property (iii) above.%
         We divide the integral above into the sum of two
         integrals $I_1+I_2$ according to the following partition of $[-1/2,1/2]^2$:
          \begin{equation} \label{part2_le split in le main proof}
        \Omega_1 = [-(k+1)\delta, (k+1)\delta]^2 \quad \text{and} \quad \Omega_2 = \Big[-\frac{1}{2},\frac{1}{2}\Big]^2 \backslash \Omega_1.
    \end{equation}
    To check that $I_2$ is finite, we use Remark \ref{part2_delta and gamma remark}$(c)$ and property (i) above to get
     \begin{equation*}
      \int_{-\delta}^\delta \iint_{\Omega_2} \frac{\abs{\Delta_h^k \big(\Phi e^{2\pi i\Psi}\big)}^2}{h^{1+r}}
      \dif x \dif y \dif h \leq \int_{-\delta}^\delta h^{2k-1-r} \dif h
      \sup_{(x,y)
          \in \Omega_2} \sup_{\abs{x-\xi} \leq k \delta} \abs{\big(\Phi e^{2\pi i\Psi}\big)_x^{(k)}
          (\xi,y)}, \\
         \end{equation*}
 where the supremum exists and is finite, and so $I_2$ is finite since $2k>r$
 (see \eqref{part2_e:who are r',s': first part} and \eqref{part2_who is k, first part}).
    That $I_1$ converges follows from Lemma \ref{part2_full
lemma}. To see this, first note that \eqref{part2_who is delta, first
part} and property (iii) above ensure that
\[\Delta_h^k \big(\Phi e^{2\pi i\Psi}\big) =
    \Delta_h^k F_{\alpha,\beta,\gamma}, \qquad \qquad \forall
    (x,y)\in \Omega_1, \quad h\in [-\delta,\delta].
    \]
In addition, the conditions
\eqref{part2_e:who are r',s': first part} and  \eqref{part2_choice of alpha and beta 1} imply that $2\alpha +
\alpha/\beta + 1 = r'=r+\epsilon$. So the choices of $k$ and
$\gamma$ ensure that Lemma \ref{part2_full lemma} can be applied (see \eqref{part2_who is k, first part} and \eqref{part2_who is gamma,
first part}).

    The system $G(g,1,1)$ is a Bessel system in $L^2(\R)$, see Remark \ref{bessel remark}.
    It remains to be checked that $G(g,1,1)$ is an exact $(C_q)$-system in $L^2(\R)$.
    By Lemma \ref{l:basic lemma, function version}, it is enough to show
    that ${1}/{\abs{Zg}^{2}}$ is in $L^{\frac{q}{q-2}}([0,1]^2)$.
    This is equivalent to
    \begin{equation}
    \label{part2_condition for cq}
         \frac{1}{\abs{\Phi(x,y)}^2} \in L^{\frac{q}{q-2}}([-1/2,1/2]^2).
    \end{equation}
    Note that properties (ii) and (iii) above imply that on $[-1/2,1/2]^2$ we have
    \begin{equation*}
        \abs{\Phi(x,y)}
        \geq C (\abs{x}^{\alpha/\gamma} + \abs{y}^{\beta/\gamma})^{\gamma}
        \geq C (\abs{x}^{\alpha/\beta} + \abs{y})^{\beta}.
    \end{equation*}
    So \eqref{part2_condition for cq} follows from  \eqref{part2_cond on r' and
    s'} and \eqref{part2_choice of alpha and beta 1}  in a direct computation. This completes the proof for the
    first part.


\subsection{Proof for the second set of conditions} Next, let $r<s<\infty$ and assume that the inequalities \eqref{case one for part b} hold.
 (The case $s=\infty$ will be dealt with separately). Choose $\epsilon>0$ small enough for the numbers
    \begin{equation}
    \label{part2_who are r' and s'}
        r' = r + \epsilon \quad \text{and} \quad s' = s + \epsilon
    \end{equation}
   to satisfy
    \begin{equation}
     \label{part2_first cond on r' and s'}
        \frac{3q-2}{q+2} \cdot \frac{1}{r'} + \frac{1}{s'} > 1
    \end{equation}
    and
    \begin{equation}
     \label{part2_second cond on r' and s'}
     \frac{1}{r'} + \frac{3}{s'} < 1.
     \end{equation}
Let the numbers $\alpha, \beta$, $k$ and $\gamma$ be as defined in
\eqref{part2_choice of alpha and beta 1},
     \eqref{part2_who is k, first part} and \eqref{part2_who is gamma, first part} respectively.

      Our objective   is to construct a
    function $\Upsilon$, which on $[-1/2,1/2)^2$
    satisfies
      \begin{equation} \label{part2_definition of upsilon}
      \Upsilon (x,y)=\Theta (x,y)-\Theta (-x,y)e^{2\pi i y}
      \end{equation}
	for some function $\Theta$, in such a way that the function
    \begin{equation} \label{part2_theta}
         G(x,y) := \Upsilon(x-1/2,y-1/2)
    \end{equation}
    is a quasi-periodic function with the desired smoothness properties and zero at $(1/2,1/2)$.

    We begin by constructing the function $\Theta$ on $[-1/2,1/2)^2$  in two steps. First, define the function
    (see Figure \ref{part2_figure theta})
    \begin{equation*}
        \Theta_0(x,y) := \left\{ \begin{array}{cl}
        \rho(x) \Big[ \big( 2(-x)^{\alpha/\gamma} + \abs{y}^{\beta/\gamma}  \big)^\gamma + 1 \Big] + 1 - \rho(x) & \quad \text{for} \quad -1/2 \leq x < 0,
        \\ \rho(x) \Big[ (x^{\alpha/\gamma} + \abs{y}^{\beta/\gamma} )^\gamma + 1 \Big] & \quad \text{for} \quad 0 \leq x < 1/2. \end{array} \right.
    \end{equation*}
    By  Lemma \ref{part2_modulus lemma} the function $\Theta_0$ belongs to $C^k((-1/2,1/2)^2 \backslash (0,0))$. In addition, it has the following properties:
     \begin{itemize}
        \item[(i)]$
        \Theta_0(x,y) = \left\{ \begin{array}{cl} 1 & \quad \text{for} \quad x \in [-1/2,-2\eta], \\
        0 & \quad \text{for} \quad  x\in [2\eta,1/2),  \end{array} \right.
    $\item[(ii)] It is bounded from below on $W_1=[-\eta, \eta] \times [-1/2, 1/2)$.
        \item[(iii)] The difference $\abs{\Theta_0(-x,y) - \Theta_0(x,y)}$ is bounded from below on the set
        $W_2=[-1/2,1/2)^2\setminus W_1$.
        \item[(iv)] On $[-\eta, \eta]^2$, it is equal to
        $f_{\alpha,\beta,\gamma}(x,y) + 1$, where the function $f_{\alpha,\beta,\gamma}$ is defined in \eqref{part2_definition of f} with a fixed $a=2^{\gamma/\alpha}$.
    \end{itemize}
        Next, we wish to preserve these properties for the function $\Theta$ while adding the additional condition
        \begin{itemize}
        \item[(v')]
        $\Theta(x,-y)=\Theta(x,y)\quad$ for $\quad y\in [2\eta,1/2)$.
        \end{itemize}
     To this end, denote by $\nu(t)$ a function in $C^\infty([-1/2,1/2))$ that satisfies
     \begin{equation*}
        \nu(t) = \left\{ \begin{array}{cl} 1 & \quad \text{for} \quad t \in [-1/2,-2\eta], \\
        0 & \quad \text{for} \quad  t\in [2\eta,1/2),  \end{array} \right.
    \end{equation*}
	 is bounded from below on $J=[-\eta, \eta]$, and for which $\abs{\nu(-t) - \nu(t)}$ is bounded from below
    on $[-1/2,1/2)\setminus J$.

    With this, we make the following definition for $(x,y) \in [-1/2,1/2)^2$ (see Figure \ref{part2_figure theta}):
        \begin{figure}
    	\includegraphics{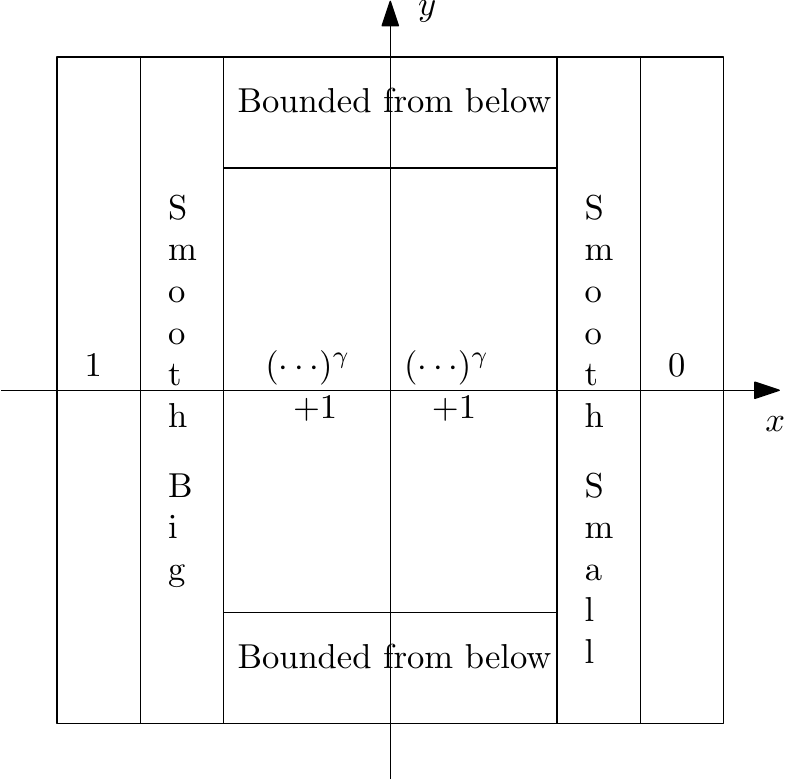}
    	\caption{Illustration of the functions $\Theta_0$ and $\Theta$ on $[-1/2,1/2)^2$.}
		\label{part2_figure theta}
    \end{figure}
    \begin{equation*}
        \Theta(x,y)  :=  \rho(y) \Theta_0(x,y) + \big(1 - \rho(y)\big)
        \nu(x).
    \end{equation*}
    One can easily verify that with this interpolation the property (v'), as well as the properties corresponding to (i)--(iv), hold for $\Theta$. For the function $\Theta$, we refer to these properties as (i')--(v').

    Define the function $\Upsilon$ on $[-1/2,1/2)^2$ by \eqref{part2_definition of upsilon},
and extend it to $\R^2$ according to the rules
\begin{equation}
      \label{part2_almost quasi periodic 2}
\Upsilon(x,y+1)=\Upsilon(x,y), \quad \Upsilon(x+1,y)=-e^{2\pi i
y}\Upsilon(x,y), \quad \qquad \forall (x,y)\in \R^2.
\end{equation}
The function $\Upsilon$ is continuous on $\R^2$ since, as follows
from property $(i')$,
\begin{equation*}
\lim_{x \rightarrow \frac{1}{2}^-}\Upsilon(x,y)=-e^{2\pi i
y}=-e^{2\pi i y}\Upsilon(-1/2,y), \qquad \forall y\in [-1/2,1/2),
\end{equation*}
while property (v') implies
\begin{equation*}
\lim_{y \rightarrow \frac{1}{2}^-}\Upsilon(x,y)=\Upsilon(x,-1/2), \qquad \forall x\in [-1/2,1/2).
\end{equation*}
With a more careful use of properties (i') and (v') one can
verify that in fact $\Upsilon\in C^k(\R^2 \setminus \Z^2)$.

    Let $G$ be the function defined in \eqref{part2_theta}. In
    particular, the conditions above imply that $G$ is bounded on
    $\R^2$. Moreover, the function $G$ is quasi-periodic, as
    follows from the condition \eqref{part2_almost quasi periodic 2} on $\Upsilon$. As
    above, this implies that there exists a function $g\in L^2(\R)$ such that
    $Zg=G$ on $\R^2$. We prove that $g$ satisfies the requirements of
    Theorem \ref{main result}.

 As in the first part, to see that $g$ has the required time--frequency localisation, we bound the integrals \eqref{part2_what has to be finite} using the partition  $\Omega_1 \cup \Omega_2$ given in \eqref{part2_le split in le main proof}. The estimates over $\Omega_2$ follow exactly as before, while the estimate for $\Delta_h$ over $\Omega_1$ follows from the inequality
     \begin{align*}
        \abs{\Delta_h^k \Upsilon(x,y)}^2 &\leq   2 \big( \abs{\Delta_h^k \Theta(x,y)}^2 + \abs{\Delta_h^k \Theta(-x,y)}^2  \big),
    \end{align*}
    in addition to Lemma \ref{part2_modulus lemma}.
	The estimate for $\Gamma_h$ can be obtained in essentially the same way, using Remark \ref{bonus lemma}$(b)$ to compensate for the additional exponential factor.

    Again, the system $G(g,1,1)$ is a Bessel system in $L^2(\R)$, see Remark \ref{bessel remark}, and so
   it remains to be checked that $G(g,1,1)$ is an exact $(C_q)$-system. As in the first part of the proof,
     it suffices to check that
     \begin{equation}\label{part2_ooooof}
     \frac{1}{\abs{\Upsilon}^2} \in L^{q/(q-2)}([-1/2,1/2]^2).
     \end{equation}
     We claim that on $[-1/2,1/2]^2$ we have $\abs{\Upsilon(x,y)}^2 >C (\abs{x}^{\alpha} +
    |y|)^2$. This,
    combined with \eqref{part2_choice of alpha and beta 1} and \eqref{part2_first cond on r' and
    s'}, implies \eqref{part2_ooooof} in a direct computation which completes the proof.
    So, to verify the estimate above, first calculate
    \begin{align*}
        \abs{\Upsilon(x,y)}^2 &=  \abs{\Theta(x,y) - \Theta(-x,y) \e^{2\pi \im y} }^2 \\
        &= (\underbrace{ \Theta(x,y) - \Theta(-x,y) }_{(I)} )^2+ 4\underbrace{
        \Theta(x,y) \Theta(-x,y) \sin^2 \pi y}_{(II)}.
    \end{align*}
    Now, by property (iii'), it follows that $(I)$
    is bounded from below on $W_2$. In the same way $(II)$ is bounded
    from below
    on $W_1\setminus[-\eta,\eta]^2$ due to property (ii').
    This leaves the region $[-\eta, \eta]^2$, where, by the above calculation and (iv'), we have
    \begin{equation*}
        \abs{\Upsilon(x,y)}^2 
        \geq  C \Big\{ \underbrace{ \Big( \big( 2\abs{x}^{\alpha/\gamma}  + \abs{y}^{\beta/\gamma}  \big)^\gamma -  \big( \abs{x}^{\alpha/\gamma}  + \abs{y}^{\beta/\gamma}  \big)^\gamma  \Big)^2 + y^2 }_{(III)} \Big\}.
    \end{equation*}
    By \eqref{part2_choice of alpha and beta 1} and \eqref{part2_second cond on r' and s'}, we have $\beta >1$, and so $y^2 \geq y^{2\beta}$. In addition, we note that since $0<\gamma<1$,  we have
    $a^\gamma + b^\gamma \geq (a+b)^\gamma$ for positive $a$ and $b$. Using these facts, and \eqref{secret trick}, we obtain 
        \begin{align*}
            (III)
            &\geq  \frac{1}{2} \left[ \Big( \big( 2\abs{x}^{\alpha/\gamma}  + \abs{y}^{\beta/\gamma}  \big)^\gamma -  \big( \abs{x}^{\alpha/\gamma}  + \abs{y}^{\beta/\gamma}  \big)^\gamma  \Big)^2 + y^{2\beta} + y^2 \right]\\
            &\geq  \frac{1}{4} \left[ \Big( \big( 2\abs{x}^{\alpha/\gamma}  + \abs{y}^{\beta/\gamma}  \big)^\gamma + (y^{\beta/\gamma})^\gamma -  \big( \abs{x}^{\alpha/\gamma}  + \abs{y}^{\beta/\gamma}  \big)^\gamma  \Big)^2 +y^2 \right]\\
            &\geq  \frac{1}{4} \left[\Big( \big( 2\abs{x}^{\alpha/\gamma}  + 2\abs{y}^{\beta/\gamma}  \big)^\gamma  -  \big( \abs{x}^{\alpha/\gamma}  + \abs{y}^{\beta/\gamma}  \big)^\gamma  \Big)^2 +y^2\right]\\
            &
            \geq C\left[\abs{x}^{2\alpha} + y^2\right]
            \geq C (\abs{x}^{\alpha} + |y|)^2.
    \end{align*}

We turn to the case $s=\infty$. If the point
$(1/r,0)$ is above the curve $\Gamma_q$, then there exists $r'$
which satisfies $r<r'<(3q-2)/(q+2)$. Set $\alpha=(r'-1)/2$. The
construction of the required example can be done in the same way
as above with the following modification of the function
$\Theta$:
\begin{equation*}
       \Theta(x,y)= \Theta_0(x,y) = \left\{ \begin{array}{cl}
        \rho(x)(2(-x)^\alpha  + 1) +1 - \rho(x) & \quad \text{for} \quad -1/2 \leq x < 0,
        \\ \rho(x)( x^\alpha + 1) & \quad \text{for} \quad 0 \leq x < 1/2, \end{array}
        \right.
    \end{equation*}
    and a corresponding change in the formulation and proof of
    Lemma 10. Note that in this case, the functions $g$ satisfying $Zg=G$, can be
    given explicitly. These functions are essentially the
    same as the functions $g_{\alpha}$ constructed in \cite{heil_powell09}, Section 6.

\qed

\section{Concluding remarks} \label{remarks section}
\begin{remark}
It is well-known that if a system $G(g,1,1)$ is a frame in
$L^2(\R)$ (i.e., a Bessel $(C_2)$-system), then it is also exact in
the space, and therefore   a Riesz basis. For general
$(C_q)$-systems, however, this is not the case: There exists a
system $G(g,1,1)$ which is a (Bessel) $(C_q)$-system, for every
$q>2$, but  is not exact. This can be shown in much the
same way as  \cite[Theorem 2]{nitzan2009quasi}.
\end{remark}

\begin{remark}
	The systems $G(g,1,1)$ constructed in Section \ref{final section} prove the following claim:
	For every $q_0 > 2$ there exists a system that is a (Bessel) $(C_q)$-system whenever $q>q_0$, but is not such a system for $q<q_0$.
	Indeed, the first part follows from the construction and the latter part follows by Theorem \ref{main result}$(a)$.
\end{remark}

\begin{remark} \label{last remark}
	Fix $q \geq 2$. It follows by Theorem 2$(a)$ that if the point $(1/r, 1/s)$ is below the curve $\Gamma_q$, then a function
	$g$, for which both the integrals in  \eqref{e:bl cond r-s} converge, cannot generate a $(C_q)$-system. By Theorem \ref{main result}$(b)$,
	on the other hand, if $(1/r,1/s)$ is above the curve $\Gamma_q$, then  the function  can generate a $(C_q)$-system.
	However, if $(1/r,1/s)$ is on the curve $\Gamma_q$, then Theorem \ref{main result}
	does not determine whether $g$ can generate a $(C_q)$-system, or not.
	%
%
We mention this question as a possible problem
for future research.
\end{remark}

\section*{Acknowledgement}
The authors would like to thank K.~Seip for his support and encouragement.

\bibliographystyle{amsplain}

\begin{thebibliography}{10}

\bibitem{balian1981}
Roger Balian, \emph{Un principe d'incertitude fort en th\'eorie du signal ou en
  m\'ecanique quantique}, C. R. Acad. Sci. Paris S\'er. II M\'ec. Phys. Chim.
  Sci. Univers Sci. Terre \textbf{292} (1981), no.~20, 1357--1362.

\bibitem{benedetto_czaja_gadzinski_powell2003}
John~J. Benedetto, Wojciech Czaja, Przemyslaw Gadzi{\' n}ski, and Alexander~M.
  Powell, \emph{The {B}alian-{L}ow theorem and regularity of {G}abor systems},
  J. Geom. Anal. \textbf{13} (2003), no.~2, 239--254.

\bibitem{benedetto_czaja_powell2006}
John~J. Benedetto, Wojciech Czaja, and Alexander~M. Powell, \emph{An optimal
  example for the {B}alian-low uncertainty principle}, SIAM J. Math. Anal.
  \textbf{38} (2006), no.~1, 333--345 (electronic).

\bibitem{benedetto_czaja2006}
John~J. Benedetto, Wojciech Czaja, Alexander~M. Powell, and Jacob Sterbenz,
  \emph{An endpoint {$(1,\infty)$} {B}alian-{L}ow theorem}, Math. Res. Lett.
  \textbf{13} (2006), no.~2-3, 467--474.

\bibitem{daubechies1990}
Ingrid Daubechies, \emph{The wavelet transform, time-frequency localization and
  signal analysis}, IEEE Trans. Inform. Theory \textbf{36} (1990), no.~5,
  961--1005.

\bibitem{daubechies_janssen93}
Ingrid Daubechies and Augustus J. E.~M. Janssen, \emph{Two theorems on lattice
  expansions}, IEEE Trans. Inform. Theory \textbf{39} (1993), no.~1, 3--6.

\bibitem{dorfler2001}
Monika D{\"o}rfler, \emph{{Time-frequency analysis for music signals: A
  mathematical approach}}, J. New Mus. Res. \textbf{30} (2001), no.~1, 3--12.

\bibitem{gabor1946theory}
Dennis Gabor, \emph{{Theory of communication}}, J. Inst. Elec. Eng. \textbf{93}
  (1946), 429--457.

\bibitem{gautam08}
Sushrut~Z. Gautam, \emph{A critical-exponent {B}alian-{L}ow theorem}, Math.
  Res. Lett. \textbf{15} (2008), no.~3, 471--483.

\bibitem{grochenig96}
Karlheinz Gr{\"o}chenig, \emph{An uncertainty principle related to the
  {P}oisson summation formula}, Studia Math. \textbf{121} (1996), no.~1,
  87--104.

\bibitem{grochenig01}
Karlheinz Gr\"{o}chenig, \emph{Foundations of time-frequency analysis},
  Birkhauser, Boston, 2001.

\bibitem{grochenig1999}
Karlheinz. Gr{\"o}chenig and Christopher Heil, \emph{Modulation spaces and
  pseudodifferential operators}, Integral Equations Operator Theory \textbf{34}
  (1999), no.~4, 439--457.

\bibitem{heil07}
Christopher Heil, \emph{History and evolution of the density theorem for
  {G}abor frames}, J. Fourier Anal. Appl. \textbf{13} (2007), no.~2, 113--166.

\bibitem{heil_powell09}
Christopher Heil and Alexander~M. Powell, \emph{Regularity for complete and
  minimal {G}abor systems on a lattice}, Illinois J. Math., to appear.

\bibitem{heil_powell06}
Christopher Heil and Alexander~M. Powell, \emph{Gabor {S}chauder bases and the
  {B}alian-{L}ow theorem}, J. Math. Phys. \textbf{47} (2006), no.~11, 1--21.

\bibitem{kozek1998}
Werner Kozek, \emph{Adaptation of {W}eyl-{H}eisenberg frames to underspread
  environments}, Gabor analysis and algorithms, Appl. Numer. Harmon. Anal.,
  Birkh\"auser Boston, Boston, MA, 1998, pp.~323--352.

\bibitem{low1985}
Francis~E. Low, \emph{Complete sets of wave packets}, A passion for physics -
  essays in honor of Geoffrey Chew (C.~DeTar et~al., ed.), World Scientific,
  Singapore, 1985, pp.~17--22.

\bibitem{nitzan09}
Shahaf Nitzan, \emph{Frame-type systems}, Ph.D. thesis, Tel-Aviv University,
  2009, Doctoral Thesis.

\bibitem{nitzan2009quasi}
Shahaf Nitzan and Alexander Olevskii, \emph{Quasi-frames of translates}, C. R.
  Math. Acad. Sci. Paris \textbf{347} (2009), no.~13-14, 739--742.

\bibitem{nitzan_olevskii07}
Shahaf Nitzan-Hahamov and Alexander Olevskii, \emph{Sparse exponential systems:
  completeness with estimates}, Israel J. Math. \textbf{158} (2007), 205--215.

\bibitem{ramanathan_steger95}
Jayakumar Ramanathan and Tim Steger, \emph{Incompleteness of sparse coherent
  states}, Appl. Comput. Harmon. Anal. \textbf{2} (1995), no.~2, 148--153.

\bibitem{seip92}
Kristian Seip, \emph{Density theorems for sampling and interpolation in the
  {B}argmann-{F}ock space. {I}}, J. Reine Angew. Math. \textbf{429} (1992),
  91--106.

\bibitem{stein70}
Elias~M. Stein, \emph{Singular {I}ntegrals and {D}ifferentiability {P}roperties
  of {F}unctions}, Princeton Mathematical Series, No. 30, Princeton University
  Press, Princeton, N.J., 1970.

\bibitem{strohmer2006}
Thomas Strohmer, \emph{Pseudodifferential operators and {B}anach algebras in
  mobile communications}, Appl. Comput. Harmon. Anal. \textbf{20} (2006),
  no.~2, 237--249.

\bibitem{young01}
Robert~M. Young, \emph{An {I}ntroduction to {N}onharmonic {F}ourier {S}eries},
  revised first ed., Academic Press Inc., San Diego, CA, 2001.

\end{thebibliography}
\providecommand{\bysame}{\leavevmode\hbox to3em{\hrulefill}\thinspace}
\providecommand{\MR}{\relax\ifhmode\unskip\space\fi MR }
\providecommand{\MRhref}[2]{%
  \href{http://www.ams.org/mathscinet-getitem?mr=#1}{#2}
}
\providecommand{\href}[2]{#2}

\end{document}